\documentclass[12pt]{amsart}
\usepackage[utf8]{inputenc}
\usepackage[T1]{fontenc}
\usepackage[british]{babel}

\usepackage{amsbsy}
\usepackage{amscd} 
\usepackage{amsfonts}
\usepackage{amsgen} 
\usepackage{amsmath}
\usepackage{amsopn} 
\usepackage{amssymb}
\usepackage{amstext}
\usepackage{amsthm} 
\usepackage{amsxtra}
\usepackage{bbold}
\usepackage{enumerate}
\usepackage{MnSymbol}
\usepackage{tikz}

\theoremstyle{plain} 

\newtheorem{theoremintro}{Theorem}

\newtheorem{theoremcounter}{theoremcounter}[section]

\newtheorem{corollary}[theoremcounter]{Corollary}
\newtheorem{lemma}[theoremcounter]{Lemma}

\newtheorem{theorem}[theoremcounter]{Theorem}

\theoremstyle{definition}
\newtheorem{definition}[theoremcounter]{Definition}
\newtheorem{example}[theoremcounter]{Example}
\newtheorem{remark}[theoremcounter]{Remark}

\numberwithin{equation}{section}

\newcommand{\lspan}{\ensuremath{\mathop{\mathrm{span}}}}
\newcommand{\Homm}{\ensuremath{\mathrm{Hom}}}
\newcommand{\ot}{\ensuremath{\otimes}}
\newcommand{\sSym}{\mathrm{sS}}
\newcommand{\Sym}{\mathrm{S}}
\newcommand{\ind}{\ensuremath{\mathop{\mathrm{ind}}}}
\newcommand{\ke}{\ensuremath{\mathop{\mathrm{ker}}}}
\newcommand{\NN}{\mathbb N}
\newcommand{\ZZ}{\mathbb Z}

\newcommand{\CC}{\mathbb C}
\newcommand{\cC}{\mathcal C}

\newcommand{\cR}{\mathcal R}

\newcommand{\idpart}{|}
\newcommand{\paarpart}{\sqcup}

\newcommand{\baarpartbaustein}{\rotatebox{180}{$\sqcap$}}
\newcommand{\baarpart}{
\mathrel{\vcenter{\offinterlineskip \hbox{$\baarpartbaustein$}}}}
\newcommand{\primarypart}{
\mathrel{\vcenter{\offinterlineskip
\hbox{$\baarpart$} \vskip -1.3ex \hbox{\hskip1.3ex$/$\hskip-1.2ex$-$} \vskip -1.2ex \hbox{\hskip2.2ex $\sqcap$}}}~}

\DeclareMathOperator{\Hom}{Hom}

\begin{document}
\title[Intertwiner spaces of non-easy group-theoretical quantum groups]{The intertwiner spaces of non-easy group-theoretical quantum groups}
\author[Laura Maaßen]{Laura Maaßen}
\address{Laura Maaßen, RWTH Aachen University, Lehrstuhl D für Mathematik, Pontdriesch 14-16, 52062 Aachen, Germany}
\email{laura.maassen@rwth-aachen.de}
\date{\today}

\begin{abstract}
In 2015, Raum and Weber gave a definition of group-theoretical quantum groups, a class of compact matrix quantum groups with a certain presentation as semi-direct product quantum groups, and studied the case of easy quantum groups.
In this article we determine the intertwiner spaces of non-easy group-theoretical quantum groups. We generalise group-theoretical categories of partitions and use a fiber functor to map partitions to linear maps which is  slightly different from the one for easy quantum groups. We show that this construction provides the intertwiner spaces of group-theoretical quantum groups in general. 
\end{abstract}

\maketitle

\section*{Introduction}
In 1987, Woronowicz \cite{Wo87} introduced \emph{compact matrix quantum groups} generalising the theory of compact Lie groups $G\subseteq \CC^{n\times n}$ to a non-commutative setting. Examples are provided by the quantum analogues $S_n^+$ and $O_n^+$ of the symmetric group $S_n\subseteq \CC^{n\times n}$ and the orthogonal group $O_n\subseteq \CC^{n\times n}$ defined by Wang in 1995 and 1998 \cite{Wa95,Wa98}. A compact matrix quantum group consists of a C*-algebra $A$ generated by the entries of a matrix $u=(u_{i,j})$, called fundamental corepresentation, and a *-homomorphism $\Delta:A\to A\ot A$, called comultiplication, and it satisfies certain dualised group properties (see Def.\ref{def::CMQG}). By a Tannaka-Krein type result of Woronowicz \cite{Wo88}, compact matrix quantum groups can be fully recovered from their intertwiner spaces, and hence any tensor category with duals (see Def.\ref{def::tensor_cat}) gives rise to a compact matrix quantum group. 

Based on this, in 2009 Banica and Speicher \cite{BS09} defined a class of compact matrix quantum groups $S_n \subseteq G \subseteq O_n^+$, called \emph{orthogonal easy quantum groups}, through a combinatorial structure of their intertwiner spaces. Their construction works as follows. They consider \emph{categories of partitions}, which are sets of set partitions closed under certain operations, and a fiber functor $p\mapsto T_p^{(n)}$ mapping partitions to linear maps. The linear span of the image of a category of partitions under this functor forms a tensor category with duals and hence gives rise to a compact matrix quantum group.

Related to the classification of all orthogonal easy quantum groups, in 2015 Raum and Weber \cite{RW15} introduced \emph{group-theoretical quantum groups} as orthogonal compact matrix quantum groups whose squared entries of the fundamental corepresentation $u_{ij}^2$ are central projections. In other words, they consider $G\subseteq H_n^{[\infty]}\subseteq H_n^+$ where $H_n^+$ is the free hyperoctahedral quantum group and $C(H_n^{[\infty]}) = C^*(\ZZ_2^{*n}) \Join C(S_n)$ (see \cite{Bi04}). Raum and Weber showed that any group-theoretical quantum group is isomorphic to a semi-direct product quantum group $C^*(\ZZ_2^{*n}/N) \Join C(\Sym_n)$, where $N\unlhd \ZZ_2^{*n}$ is an $\Sym_n$-invariant normal subgroup of $\ZZ_2^{*n}$. Moreover, they showed that a group-theoretical quantum group is easy if and only if $N$ is strongly $\Sym_n$-invariant and determined the corresponding categories of partitions. Hence the structure of the intertwiner spaces of group-theoretical easy quantum groups is known.

The aim of this article is to describe the intertwiner spaces of group-theoretical quantum groups in general, in particular those of non-easy group-theoretical quantum groups. For this purpose we define a generalisation of group-theoretical categories of partitions, called skew categories of partitions, which are still based on set partitions but closed under slightly different operations than categories of partitions. We define a fiber functor, denoted by $p\mapsto \widehat{T}_p^{(n)}$, to associate skew categories of partitions to tensor categories with duals. This functor is different from $p\mapsto T_p^{(n)}$. Then we can show that skew categories carry the ``group-theoretical'' structure we are looking for and can be characterised via their corresponding tensor categories.

\begin{theoremintro}[Theorem \ref{ThMain1}, Theorem \ref{ThMain2}]
Let $\cR\subseteq P$ be closed under rotation. Then the following are equivalent:
\begin{enumerate}
\item $\cR$ is a skew category of partitions.
\item We have $\cR= \{p\mid p \text{ is a rotation of  } \ke(\textbf{i}) \text{ for some } a_{\textbf{i}} \in F_{\infty}(\cR)\}$ and \linebreak $F_{\infty}(\cR):=\Sym_{\infty}(\{a_{\ind(p)} \mid k\in \NN_0, p\in \cR(k,0)\}) \unlhd \ZZ_2^{*\infty}$ is an ($\Sym_{\infty}$-invariant) normal subgroup of $\ZZ_2^{*\infty}$. Here $\ke(\textbf{i})$ denotes the partition with just upper points which fits exactly to the labelling $\textbf{i}$ and $a_{\ind(p)} \in \ZZ_2^{*\infty} =\langle a_1,a_2, \ldots \rangle$ is obtained by labelling the blocks of $p$ with the letters $a_1,a_2\ldots$ in ascending order.
\item $\lspan \{\widehat{T}_p^{(n)} \mid p\in \cR(k,l)\},k,l\in \NN_0$ is a tensor category with duals for all $n\in \NN$.
\end{enumerate}
\end{theoremintro}

Applying the previous results and using the semi-direct product structure of group-theoretical quantum groups, allows us to prove that skew categories are the right analogues of categories of partitions for
most of the group-theoretical quantum groups.
\begin{theoremintro}[Theorem \ref{ThMain3}]
Let $S_n \subseteq G \subseteq O_n^+$ be a (not necessarily easy) compact matrix quantum group in its maximal version. Then $G$ is a group-theoretical quantum group $C^*(\ZZ_2^{*n}/N) \Join C(\Sym_n)$ with $\langle \langle \Sym_{\infty}(N) \rangle \rangle_{\ZZ_2^{*\infty}} \cap \ZZ_2^{*n} = N$ if and only if there exists a skew category of partitions $\cR$ such that 
$\Hom_G(k,l)=\lspan \{\widehat{T}_p^{(n)} \mid p\in \cR(k,l)\}$
for all $k,l\in \NN_0$. In this case, we have 
$\cR =\langle \ke(\textbf{i}) \mid a_{i_1}\cdot \ldots \cdot a_{i_k} \in F_n(\cR) \rangle_{skew}$
and 
$C(G) \cong C^*(\ZZ_2^{*n}/F_n(\cR)) \Join C(\Sym_n)$
with $F_n(\cR)=F_{\infty}(\cR)\cap \ZZ_2^{*n}$.
\end{theoremintro}
Note that, in particular, any skew category of partitions $\cR$ gives rise to a series of group-theoretical quantum groups $G_n$ with
\[ C(G_n) = C^*(\ZZ_2^{*n}/(F_\infty(\cR)\cap \ZZ_2^{*n})) \Join C(\Sym_n) .\]
Group-theoretical quantum groups that do not satisfy the technical assumption on the corresponding normal subgroup are treated in Remark \ref{rem::gen_case}. Hence, all compact matrix quantum groups $S_n \subseteq G\subseteq H_n^{[\infty]}\subseteq H_n^+$ (not necessarily easy quantum groups) are classified. Thus part of the open problem to classify all compact matrix quantum groups $S_n \subseteq G\subseteq H_n^+$, stated for example by Banica in \cite{Ba18}, is solved. 

In Section 1 we will recall some basic definitions for partitions and introduce skew categories of partitions. Section 2 provides some technical background and the proof of the first equivalence of Theorem 1. In Section 3 we introduce the functor $p\mapsto \widehat{T}_p^{(n)}$ and prove the other part of Theorem 1. We start Section 4 by recalling some basic definitions for compact matrix quantum groups and easy quantum groups and we summarise the results of Raum and Weber. Then we can finally prove Theorem 2 and have a look at some corollaries. Section 5 provides an example of a non-easy group-theoretical quantum group.

\section*{Acknowledgements}
The author is supported by an RWTH Scholarship for Doctoral Students and the Integrated Research Training Group (IRTG) of the SFB-TRR 195 ``Symbolic Tools in Mathematics and their Application''. This article contributes to the project ``I.13.-Computational classification of orthogonal quantum groups'' of the SFB-TRR 195. The author thanks her PhD advisors prof. Gerhard Hiss and prof. Moritz Weber for many more than helpful discussions and comments. The author also thanks the anonymous referee for the insightful comments, especially on Theorem \ref{ThMain1}. This article is part of the author's PhD thesis.

\section{Skew categories of partitions}
We recall the definition of categories of partitions and we introduce the new concept of skew categories of partitions.

\subsection{Partitions and operations}
At first we recall some basic definitions (see \cite{BS09,RW15}). 
\begin{definition} \label{def::partition}
Let $k,l\in \NN_0$. A \emph{partition} $p\in P(k,l)$ is a partition into disjoint, non-empty subsets of the set $\{ 1,\ldots ,k,1',\ldots ,l' \}$. These subsets are called the \emph{blocks of $p$} and we denote their number by $\text{bl}(p)$. Moreover, we put $P:= \bigcup_{k,l\in \NN_0} P(k,l)$ and
for a subset $D \subseteq P$ and $n\in \NN$ we define
\[ D_n := \{ p\in D \mid p \text{ has at most } n \text{ blocks}\} .\]
We can picture every partition $p\in P(k,l)$ as $k$ upper and $l$ lower points, where all points in the same block of $p$ are connected by a string. 
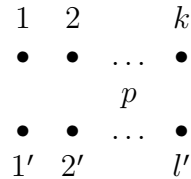
\begin{align*}
\begin{matrix}
1 & 2 &  & k \\
\bullet & \bullet & \ldots & \bullet \\
 & & p & & \\
\bullet & \bullet & \ldots & \bullet \\
1' & 2' &  & l' \\
\end{matrix}
\end{align*}
\end{definition}

As easy examples consider the following partitions: \\
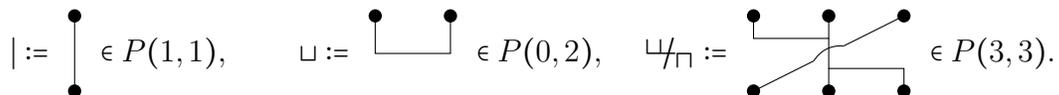
\begin{minipage}[t]{.25\linewidth}
\vspace{0pt}
\begin{tikzpicture}
\coordinate [label=left:{$\idpart:=$}](O) at (-0.2,0);
\coordinate [label=right:{$\in P(1,1),$}](O2) at (0.2,0);
\coordinate (A1) at (0,-0.5);
\coordinate (A2) at (0,0.5);

\fill (A1) circle (2.5pt);
\fill (A2) circle (2.5pt);

\draw (A1) -- (A2);
\end{tikzpicture}
\end{minipage}
\begin{minipage}[t]{.3\linewidth}
\vspace{0pt}
\begin{tikzpicture}
\coordinate [label=left:{$\paarpart:= $}](O1) at (-0.2,0);
\coordinate [label=right:{$\in P(2,0),$}](O2) at (0.7,0);
\coordinate (A1) at (0,0.5);
\coordinate (A2) at (0.5,0.5);

\coordinate (C1) at (0,0);
\coordinate (C2) at (0.5,0);

\fill (A1) circle (2.5pt);
\fill (A2) circle (2.5pt);

\draw (A1) -- (C1) -- (C2) -- (A2);
\end{tikzpicture}
\end{minipage}
\begin{minipage}[t]{.45\linewidth}
\vspace{0pt}
\begin{tikzpicture}
\coordinate [label=left:{$\primarypart := $}](O) at (-0.2,0);
\coordinate [label=right:{$\in P(3,3).$}](O2) at (1.2,0);
\coordinate (A1) at (0,-0.5);
\coordinate (A2) at (0.5,-0.5);
\coordinate (A3) at (1,-0.5);
\coordinate (A4) at (0,0.5);
\coordinate (A5) at (0.5,0.5);
\coordinate (A6) at (1,0.5);

\coordinate (C1) at (0,0.2);
\coordinate (C2) at (0.5,0.2);
\coordinate (C3) at (0.5,-0.2);
\coordinate (C4) at (1,-0.2);
\coordinate (C5) at (0.4,-0.05);
\coordinate (C6) at (0.6,0.05);

\fill (A1) circle (2.5pt);
\fill (A2) circle (2.5pt);
\fill (A3) circle (2.5pt);
\fill (A4) circle (2.5pt);
\fill (A5) circle (2.5pt);
\fill (A6) circle (2.5pt);

\draw (A4) -- (C1) -- (C2) -- (A5);
\draw (A2) -- (C3) -- (C4) -- (A3);
\draw (A1) -- (C5);
\draw (C6) -- (A6);
\draw (C5) to [bend left] (C6);
\draw (C2) -- (C3);
\end{tikzpicture}
\end{minipage} 
\ \\
For easier calculation with partitions we will associate pairs of multi-indices to partitions.

\begin{definition}{\cite[§ 2.4.]{RW15}} \label{def::ker_ind}
Let $k,l\in \NN_0$. \\
For all $\textbf{i}=(i_1,\ldots ,i_k) \in \NN^k,\textbf{j}=(j_1,\ldots ,j_l) \in \NN^l$ we define $\ke(\textbf{i},\textbf{j}) \in P(k,l)$ as the partition obtained by the fibers of the map $\phi:\{1,\ldots,k,1'\ldots,l'\} \to \NN$ with $\phi(x)=i_x$ for all $x\in \underline{k}$ and $\phi(y')=j_y$ for all $y\in \underline{l}$. For $l=0$ and $\textbf{i} \in \NN^k$ we denote $\ke(\textbf{i}):=\ke(\textbf{i},\emptyset)$.\\
Vice versa for any partition $p\in P(k,l)$ we define $\ind(p):=(\textbf{i}^{(p)},\textbf{j}^{(p)})\in \NN^k \times \NN^l$ as the lexicographic minimal element in $\NN^{k+l}$ with $\ke(\ind(p))=p$.
\end{definition}
 
For example we have $\ind(p)=((1,1,2),(2,1,1))\in \NN^3 \times \NN^3$ and $p=\ke((3,3,7),(7,$ $3,3))$ for $p=\primarypart \in P(3,3)$. 

\begin{definition}[Symmetric group]
Let $n\in \NN$. We view the \emph{symmetric group} $\Sym_n$ as the group $\{\sigma:\underline{n} \to \underline{n} \mid \sigma \text{ bijective}\}$. We define the \emph{symmetric group (on countably, but infinitely many points)} $\Sym_{\infty}$ as 
\[ \Sym_{\infty}:= \{ \sigma: \NN \to \NN \mid \sigma \text{ bijective and } |\{n\in \NN \mid \sigma(n)\neq n\}|<\infty  \} .\] 
\end{definition}

Note that $\Sym_{\infty}$ acts componentwise on the multi-indices $\NN^{k+l}$ and we have $p=\ke(\textbf{i},\textbf{j})$ if and only if $(\textbf{i},\textbf{j}) \in \Sym_{\infty}(\ind(p))$.

\bigskip
Now, we recall the operations on partitions Banica and Speicher introduced to construct easy quantum groups.

\begin{definition}[Operations on partitions {\cite[Def.1.8.]{BS09}}] \label{def::cat_operations}
Let $p\in P(k,l)$ and $q\in P(k',l')$.
\begin{enumerate}[$\bullet$]
\item The \emph{involution} $p^* \in P(l,k)$ is obtained by turning $p$ upside-down.
\item The \emph{tensor product} $p\otimes q \in P(k+k',l+l')$ is the horizontal concatenation of the partitions $p$ and $q$.
\item Let $l=k'$. Then we can consider the vertical concatenation of the partitions $p$ and $q$. We may obtain middle points which are neither connected to upper nor to lower points. Connected components of such points are called \emph{loops} and we denote their number by $l(q,p)$. The \emph{composition} $qp \in P(k,l')$ of $p$ and $q$ is the vertical concatenation, where we remove all loops.
\item One \emph{basic rotation of $p$} is obtained by turning the uttermost left leg of the upper row and putting it in front of the first leg of the lower row or just in the lower row if $p$ has no lower points. Similarly, the other basic rotations are obtained by rotating the uttermost left lower leg to the upper row, the uttermost right upper leg to the lower row or the uttermost right lower leg to the upper row. Multiple basic rotations of $p$ are called \emph{rotations of $p$}.
\end{enumerate}
\end{definition}

See \cite{We17} for examples of these operations. To construct skew categories of partitions we introduce the following slightly modified operations.

\begin{definition}[Modified operations] 
Let $p\in P(k,l)$ and $q\in P(k',l')$.
\begin{enumerate}[$\bullet$]
\item For all $(\textbf{i},\textbf{j})\in \Sym_{\infty}(\ind(p))$ and $(\textbf{f},\textbf{g})\in \Sym_{\infty}(\ind(q))$ we define the \emph{connected tensor product} $p\ot_{(\textbf{if},\textbf{jg})} q:=ker(\textbf{if},\textbf{jg}) \in P(k+k',l+l')$. 
\item Let $l=k'$. We call $p$ and $q$ \emph{compatible} if $\textbf{j}^{(p)}\in \Sym_{\infty}(\textbf{i}^{(q)})$ where $\textbf{j}^{(p)}$ and $\textbf{i}^{(q)}$ are as in Def. \ref{def::ker_ind}. If $p$ and $q$ are compatible we define the \emph{conditioned composition of $p$ and $q$} as the usual composition $qp \in P(k,l')$. 
\item Let $l=k'$ and let $p$ and $q$ be compatible. For all $(\textbf{i},\textbf{j})\in \Sym_{\infty}(\ind(p))$ and $(\textbf{j},\textbf{g})\in \Sym_{\infty}(\ind(q))$ we define the \emph{connected conditioned composition} $q\cdot_{(\textbf{i},\textbf{j},\textbf{g})} p:=\ke(\textbf{i},\textbf{g}) \in P(k,l')$.
\end{enumerate}
\end{definition}

Note that the connected tensor product $p\ot_{(\textbf{if},\textbf{jg})} q$ is equal to the usual tensor product $p\ot q$ if the multi-indices $(\textbf{i},\textbf{j}) \in \NN^{k+l}$ and $(\textbf{f},\textbf{g}) \in \NN^{k'+l'}$ have pairwise different entries. In general all connected tensor products can be obtained by taking the usual tensor product and applying several joinings of one block of $p$ with one block of $q$ where it is not allowed to connect different blocks of $p$ or different blocks of $q$. Consider the following example: 
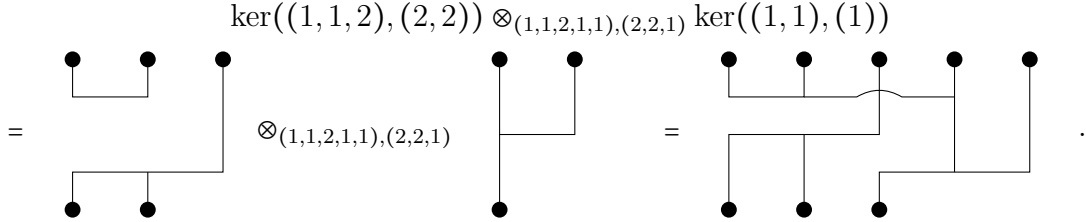
\[ \ke((1,1,2),(2,2)) \ot_{(1,1,2,1,1),(2,2,1)} \ke((1,1),(1)) \]
\begin{minipage}[t]{.35\linewidth}
\begin{tikzpicture}
\coordinate [label=left:{$=$}](O) at (-0.5,1);
\coordinate (A1) at (0,2);
\coordinate (A2) at (0.5,2);
\coordinate (A3) at (1,2);
\coordinate (B1) at (0,0);
\coordinate (B2) at (0.5,0);
\coordinate [label=right:{$\ot_{(1,1,2,1,1),(2,2,1)}$}](O) at (1.3,1);

\coordinate (C1) at (0,1.5);
\coordinate (C2) at (0,0.5);
\coordinate (C3) at (0.5,1.5);
\coordinate (C4) at (0.5,0.5);
\coordinate (C5) at (1,0.5);

\fill (A1) circle (3pt);
\fill (A2) circle (3pt);
\fill (A3) circle (3pt);
\fill (B1) circle (3pt);
\fill (B2) circle (3pt);

\draw (A1) -- (C1) -- (C3) -- (A2);
\draw (B1) -- (C2) -- (C4) -- (C5) -- (A3);
\draw (B2) -- (C4);
\end{tikzpicture}
\end{minipage}
\begin{minipage}[t]{.17\linewidth}
\hspace{1.1cm}
\begin{tikzpicture}
\coordinate (A1) at (0,2);
\coordinate (A2) at (0.5,2);
\coordinate (B1) at (0,0);

\coordinate (C1) at (0,1);
\coordinate (C2) at (0.5,1);

\fill (A1) circle (3pt);
\fill (A2) circle (3pt);
\fill (B1) circle (3pt);

\draw (A1) -- (C1) -- (C2) -- (A2);
\draw (C1) -- (B1);
\end{tikzpicture}
\end{minipage}
\begin{minipage}[t]{.5\linewidth}
\hspace{0.6cm}
\begin{tikzpicture}
\coordinate [label=left:{$=$}](O) at (-0.5,1);
\coordinate (A1) at (0,2);
\coordinate (A2) at (0.5,2);
\coordinate (A3) at (1,2);
\coordinate (A4) at (1.5,2);
\coordinate (A5) at (2,2);
\coordinate (B1) at (0,0);
\coordinate (B2) at (0.5,0);
\coordinate (B3) at (1,0);
\coordinate [label=right:{.}](O) at (2.25,1);

\coordinate (C1) at (0,1.5);
\coordinate (C2) at (0,1);
\coordinate (C3) at (0.5,1.5);
\coordinate (C4) at (0.5,1);
\coordinate (C5) at (1,1);
\coordinate (C6) at (1,0.5);
\coordinate (C7) at (1.5,0.5);
\coordinate (C8) at (2,0.5);
\coordinate (C9) at (0.8,1.5);
\coordinate (C10) at (1.2,1.5);
\coordinate (C11) at (1.5,1.5);

\fill (A1) circle (3pt);
\fill (A2) circle (3pt);
\fill (A3) circle (3pt);
\fill (A4) circle (3pt);
\fill (A5) circle (3pt);
\fill (B1) circle (3pt);
\fill (B2) circle (3pt);
\fill (B3) circle (3pt);

\draw (A1) -- (C1) -- (C3) -- (A2);
\draw (B1) -- (C2) -- (C4) -- (C5) -- (A3);
\draw (B2) -- (C4);
\draw (B3) -- (C6) -- (C7) -- (C8) -- (A5);
\draw (A4) -- (C7);
\draw (C3) -- (C9) to [bend left] (C10) -- (C11);
\end{tikzpicture}
\end{minipage}

Thus the connected tensor product allows more operations than the usual tensor product whereas the conditioned composition restricts the usual composition to compatible partitions, i.e. to those partitions whose upper resp. lower row coincide with respect to the block structure. In the following example the partitions $p_1$ and $p_2$ are compatible but the partitions $p_1$ and $p_3$ are not compatible. \\

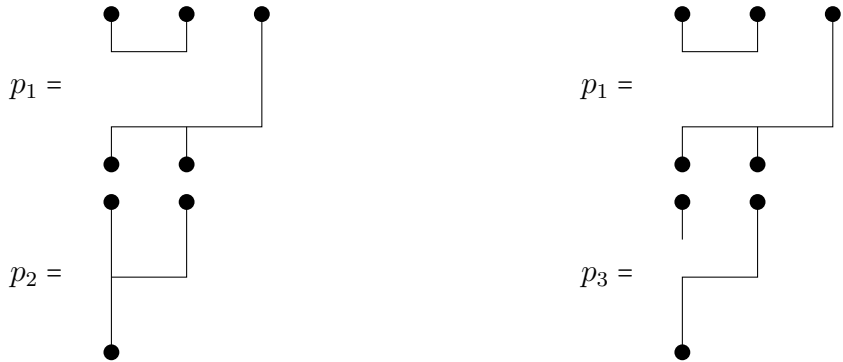
\begin{minipage}[t]{.5\linewidth}
\begin{tikzpicture}
\coordinate [label=left:{$p_1=$}](O) at (-0.5,1);
\coordinate [label=left:{$p_2=$}](O) at (-0.5,-1.5);
\coordinate (A1) at (0,2);
\coordinate (A2) at (0.5,2);
\coordinate (A3) at (1,2);
\coordinate (B1) at (0,0);
\coordinate (B2) at (0.5,0);
\coordinate (D1) at (0,-0.5);
\coordinate (D2) at (0.5,-0.5);
\coordinate (E1) at (0,-2.5);

\coordinate (C1) at (0,1.5);
\coordinate (C2) at (0,0.5);
\coordinate (C3) at (0.5,1.5);
\coordinate (C4) at (0.5,0.5);
\coordinate (C5) at (1,0.5);
\coordinate (C6) at (0,-1.5);
\coordinate (C7) at (0.5,-1.5);

\fill (A1) circle (3pt);
\fill (A2) circle (3pt);
\fill (A3) circle (3pt);
\fill (B1) circle (3pt);
\fill (B2) circle (3pt);
\fill (D1) circle (3pt);
\fill (D2) circle (3pt);
\fill (E1) circle (3pt);

\draw (A1) -- (C1) -- (C3) -- (A2);
\draw (B1) -- (C2) -- (C4) -- (C5) -- (A3);
\draw (B2) -- (C4);
\draw (D1) -- (E1); 
\draw (C6) -- (C7) -- (D2);
\end{tikzpicture}
\end{minipage}
\begin{minipage}[t]{.5\linewidth}
\begin{tikzpicture}
\coordinate [label=left:{$p_1=$}](O) at (-0.5,1);
\coordinate [label=left:{$p_3=$}](O) at (-0.5,-1.5);
\coordinate (A1) at (0,2);
\coordinate (A2) at (0.5,2);
\coordinate (A3) at (1,2);
\coordinate (B1) at (0,0);
\coordinate (B2) at (0.5,0);
\coordinate (D1) at (0,-0.5);
\coordinate (D2) at (0.5,-0.5);
\coordinate (E1) at (0,-2.5);

\coordinate (C1) at (0,1.5);
\coordinate (C2) at (0,0.5);
\coordinate (C3) at (0.5,1.5);
\coordinate (C4) at (0.5,0.5);
\coordinate (C5) at (1,0.5);
\coordinate (C6) at (0,-1.5);
\coordinate (C7) at (0.5,-1.5);
\coordinate (C8) at (0,-1);

\fill (A1) circle (3pt);
\fill (A2) circle (3pt);
\fill (A3) circle (3pt);
\fill (B1) circle (3pt);
\fill (B2) circle (3pt);
\fill (D1) circle (3pt);
\fill (D2) circle (3pt);
\fill (E1) circle (3pt);

\draw (A1) -- (C1) -- (C3) -- (A2);
\draw (B1) -- (C2) -- (C4) -- (C5) -- (A3);
\draw (B2) -- (C4);
\draw (D1) -- (C8); 
\draw (E1) -- (C6) -- (C7) -- (D2);
\end{tikzpicture}
\end{minipage}

Moreover, note that the connected conditioned composition $q\cdot_{(\textbf{i},\textbf{j},\textbf{g})} p$ is equal to the conditioned composition $qp$ if the entries of the multi-indices $\textbf{i} \in \NN^{k}$ and $\textbf{g} \in \NN^{l'}$ which do not appear in $\textbf{j} \in \NN^l$ are pairwise different. In general all connected conditioned compositions can be obtained by taking the conditioned composition and applying afterwards several joinings of one upper block with one lower block. Here it is not allowed to connect different upper blocks or different lower blocks.
We will later see that the connected conditioned composition can be obtained by several conditioned compositions and connected tensor products.

\subsection{(Skew) categories of partitions}

Based on the operations in Def. \ref{def::cat_operations} Banica and Speicher defined categories of partitions. Moreover, Raum and Weber \cite{RW15} defined group-theoretical categories of partitions which correspond to group-theoretical easy quantum groups as we shall see later.

\begin{definition}[Category of partitions {\cite[Def.6.3.]{BS09}}]
A \emph{category of partitions} is a subset $\cC \subseteq P$ containing the partitions $\sqcup$ and $|$ which is closed under involution, taking tensor products and composition. By $\langle E \rangle$ we denote the closure of $E \cup \{ \sqcap, | \}$ under involution, taking tensor products and composition for any set $E\subseteq P$. \\
A category of partitions is called \emph{group-theoretical} if it contains the partition $\primarypart$.
\end{definition}

Note that any category of partitions is closed under rotation as for example the basic rotation from the left of the upper to the lower line can be obtained by taking the tensor product with the identity partition $|$ and then compose with $\sqcap \ot | \ot \cdots \ot |$, see \cite{BS09}. Now, we consider the modified operations on partitions to give a definition of a skew category of partitions.

\begin{definition}[Skew category of partitions]
A \emph{skew category of partitions} is a subset $\cC \subseteq P$ containing the partitions $\sqcup$ and $|$ which is closed under involution, taking connected tensor products and conditioned composition. By $\langle E \rangle_{skew}$ we denote the closure of $E \cup \{ \sqcap, | \}$ under involution, taking connected tensor products and conditioned composition for any set $E\subseteq P$. \\
\end{definition}

\newpage
The following lemma shows that skew categories of partitions generalise group-theoretical categories of partitions. Recall Definition \ref{def::ker_ind}.

\begin{lemma} \label{lem::skew_cat}
\begin{enumerate}[(i)]
\item Skew categories of partitions are closed under rotation.
\item Skew categories of partitions are closed under connected conditioned composition.
\item Any skew category of partitions contains the partition $\primarypart$. 
\item Any group-theoretical category of partitions is a skew category of partitions.
\end{enumerate}
\end{lemma}

\begin{proof}
\begin{enumerate}[(i)]
\item We mimic the proof of \cite[Lemma 1.1]{TW18}. However, we need to take the special conditions of the conditioned composition into account. Let $\cR$ be a skew category of partitions and $p\in \cR(k,l)$. We define $(\textbf{f},\textbf{g}):=((1,i_1^{(p)},\ldots,i_k^{(p)}),(1,j_1^{(p)},\ldots,j_l^{(p)}))$. Then $q:= |\ot_{(\textbf{f},\textbf{g})} p \in \cR(k+1,l+1)$ is the connected tensor product connecting the partition $|$ to the uttermost left upper point of $p$ as $i_1^{(p)}=1$ by Def. \ref{def::ker_ind}. We consider also the partition $s:=\ke((i_2^{(p)},\ldots,i_k^{(p)}),(1,1,i_2^{(p)},\ldots,i_k^{(p)})) \in \cR(k-1,k+1)$ which can be obtained by several connected tensor products of $\sqcap$ and $|$. By construction, $s$ and $q$ are compatible and $qs \in \cR(k-1,l+1)$ is the basic rotation of $p$ where we rotate the uttermost left upper point to the lower row. The other basic rotations can be constructed analogously. 
\item Let $\cR$ be a skew category of partitions and let $p\in \cR(k,l)$ and $q\in \cR(l,l')$ be compatible partitions. Let $(\textbf{i},\textbf{j})\in \Sym_{\infty}(\ind(p))$ and $(\textbf{j},\textbf{g})\in \Sym_{\infty}(\ind(q))$ and put $\textbf{j}':=(j_l,\ldots,j_1) \in \NN^l$ and $\textbf{g}':=(g_{l'},\ldots,g_1) \in \NN^{l'}$. We consider the rotation $p':=\ke (\textbf{i}\textbf{j}') \in \cR(k+l,0)$ of $p$ to the upper line and the rotation $q':=\ke (\textbf{j}\textbf{g}') \in \cR(l+l',0)$ of $q$ to the upper line and put $r:=p' \ot_{(\textbf{i}\textbf{j}'\textbf{j}\textbf{g}',\emptyset)} q =\ke (\textbf{i}\textbf{j}'\textbf{j}\textbf{g}') \in \cR(k+2l+l')$. By several conditioned compositions of connected tensor products of $|$ and $\sqcap$ with $r$ we obtain $\ke (\textbf{i}\textbf{g}') \in \cR(k+l')$ and hence the connected conditioned composition $\ke (\textbf{i},\textbf{g}) \in \cR(k,l')$ lies in $\cR$.
\item We have 
\[\primarypart = (\sqcup \ot |) \ot_{(1,1,2),(2,1,1)} \sqcap.\] 
\item It is easy to see that group-theoretical categories of partitions are closed under joining blocks since they contain the partitions $\ke((1,1,2,3,\ldots,n),(2,3,\ldots,n,\linebreak 1,1))$ by composition and hence by rotation the
partitions $\ke((1,2,3,\ldots,n,1),\linebreak (1,2,3,\ldots,n,1))$. This implies that group-theoretical categories of partitions are closed under taking connected tensor products.
\end{enumerate}
\end{proof}

\newpage
\section{The group-theoretical structure of skew categories of partitions}
In this section, we will analyse the ``group-theoretical'' structure of skew categories of partitions. Following the idea of Raum and Weber, we associate words in the free group product $\ZZ_2^{*\infty}$ to partitions.

\begin{definition}{\cite[§ 4.1.]{RW15}} \label{def::Finfty}
Let $\ZZ_2^{*\infty}=\langle a_i \mid i\in \NN \rangle$ be the infinite free product of the cyclic group $\ZZ_2$. For any $\textbf{i}=(i_1,\ldots,i_k) \in \NN^k$ we denote by $a_{\textbf{i}}:=a_{i_1}\cdot \ldots \cdot a_{i_k} \in \ZZ_2^{*\infty}$ the corresponding word in $\ZZ_2^{*\infty}$ and we view any $\sigma \in \Sym_{\infty}$ as an endomorphism $\sigma\in \text{End}(\ZZ_2^{*\infty})$ via $\sigma(a_{\textbf{i}}) = a_{\sigma(\textbf{i})}$. \\
For any subset $D\subseteq P$ which is closed under rotation we define 
\[ F_{\infty}(D):=\Sym_{\infty}(\{ a_{\ind(p)} \mid k\in \NN_0, p\in D(k,0)\}) \subseteq \ZZ_2^{*\infty}.\]
\end{definition}

\begin{remark} \label{rem::Finfty}
Let $N\subseteq \ZZ_2^{*\infty}$ be a subset which is closed under conjugation and we set
\[ D:= \{p\mid p \text{ is a rotation of  } \ke(\textbf{i}) \text{ for some } a_{\textbf{i}} \in N\}.\]
Then $F_{\infty}(D) = N$ by Definition \ref{def::ker_ind}.
\end{remark}

\begin{theorem}\label{ThMain1}
Let $\cR\subseteq P$ be closed under rotation. Then $\cR$ is a skew category of partitions if and only if $F_{\infty}(\cR) \unlhd \ZZ_2^{*\infty}$ is an ($\Sym_{\infty}$-invariant) normal subgroup of $\ZZ_2^{*\infty}$ and
\[ \cR= \{p\mid p \text{ is a rotation of  } \ke(\textbf{i}) \text{ for some } a_{\textbf{i}} \in F_{\infty}(\cR)\}.\]
In particular, for any $\Sym_{\infty}$-invariant normal subgroup $N\unlhd \ZZ_2^{*\infty}$ the set 
\[ \{p\mid p \text{ is a rotation of  } \ke(\textbf{i}) \text{ for some } a_{\textbf{i}} \in N\}\]
is a skew category of partitions.
\end{theorem}

\begin{proof}
(1) Let $\cR$ be a skew category of partitions. At first we show that $F_{\infty}(\cR) \unlhd \ZZ_2^{*\infty}$ is an $\Sym_{\infty}$-invariant normal subgroup of $\ZZ_2^{*\infty}$. Recall that for any partition $p\in \cR(k,0)$ an element of $\Sym_{\infty}(\ind(p))$ is a labelling of $p$ with a maximal amount of numbers such that the labelling is constant on blocks. Now, let $p\in \cR(k,0),a_{\textbf{i}}\in \Sym_{\infty}(a_{\ind(p)})$ and $q\in \cR(k',0),a_{\textbf{i}}'\in \Sym_{\infty}(a_{\ind(q)})$.

\begin{enumerate}[$\bullet$]
\item $F_{\infty}(\cR)$ is closed under inverting elements since
\[ a_{\textbf{i}}^{-1} = (a_{i_1}\cdot \ldots \cdot a_{i_k})^{-1} = a_{i_k}\cdot \ldots \cdot a_{i_1} \in \Sym_{\infty}(a_{\ind(r)}) \in F_{\infty}(\cR) \] 
where $r\in \cR$ is the rotation of $p^*$ to the upper line.
\item $F_{\infty}(\cR)$ is closed under multiplication since
\[ a_{\textbf{i}} a_{\textbf{j}} \in \Sym_{\infty}(a_{\ind(p\ot_{(\textbf{i}\textbf{j},0)} q)}) \in F_{\infty}(\cR).\]
\item Let $j \in \NN$ and let $r\in \cR(k+2,0)$ be the partition obtained by rotating the uttermost left point of $\sqcup \ot_{((j,j,i_1,\ldots,i_k),0)} p$ to the right. It follows that
\[ a_j a_{\textbf{i}} a_j^{-1} = a_j a_{\textbf{i}} a_j \in \Sym_{\infty}(a_{\ind(r)}) \in F_{\infty}(\cR) \]
and hence $F_{\infty}(\cR)$ is closed under conjugation.
\end{enumerate}

Now, we show that $\cR=\cR'$ with 
\[ \cR':= \{p\mid p \text{ is a rotation of  } \ke(\textbf{i}) \text{ for some } a_{\textbf{i}} \in F_{\infty}(\cR)\}.\]
If $p\in \cR$ we have $a_{\ind(p)} \in F_{\infty}(\cR)$ and it follows that $p=\ke(\ind(p)) \in \cR'$. Thus we have $\cR \subseteq \cR'$. \\
So let $p\in \cR'$. As $\cR$ is a skew category of partitions, it is invariant under rotations and we may assume that $p=\ke(\textbf{i})$ for some $a_{\textbf{i}} \in F_{\infty}(\cR)$. By Definition \ref{def::Finfty} there exists a partition $q\in \cR$ such that $a_{\textbf{i}}=a_{\ind(q)}$ in $\ZZ_2^{*\infty}$. It follows from the definition of $\ZZ_2^{*\infty}$ that we can obtain $p$ from $q$ by successively either inserting the partition $\sqcup$ and possibly connect it to an existing block or removing two adjacent connected points. Thus it suffices to show that $\cR$ is closed under these operations since this would imply $p\in \cR$. For this purpose we consider $\textbf{j} \in \NN^k$ with $j_l=j_{l+1}$ for some $l<k$ and define the partition $p_{\textbf{j}}:= \ke((j_1,\ldots,j_{l-1},j_{l+2},\ldots,j_k),(j_1,\ldots, j_k))$. Since $p_{\textbf{j}}$ can be obtained by several connected tensor products of $|$ and $\sqcap$, it lies in $\cR$. Conditioned composition with partitions of this form yield exactly the required operations. Thus we can conclude that $\cR' \subseteq \cR$.

\ \\
(2) Let $\cR\subseteq P$ just be a rotation invariant subset such that $N:=F_{\infty}(\cR) \unlhd \ZZ_2^{*\infty}$ is an $\Sym_{\infty}$-invariant normal subgroup of $\ZZ_2^{*\infty}$ and such that 
$$\cR= \{p\mid p \text{ is a rotation of  } \ke(\textbf{i}) \text{ for some } a_{\textbf{i}} \in N\}.$$ We show that $\cR$ is a skew category of partitions. In the following we denote $\textbf{i}^{-1}:=(i_k,\ldots,i_1)$ for any $\textbf{i}=(i_1,\ldots,i_k)\in \NN^k$.

\begin{enumerate}[$\bullet$]
\item At first we show that $\cR$ can be expressed as
\[ \cR=\{ \ke (\textbf{i},\textbf{j}) \mid a_{\textbf{i}} a_{\textbf{j}}^{-1} \in N\}.\]
Let $\textbf{i} \in \NN^k, \textbf{j} \in \NN^l$. First, let $p=\ke (\textbf{i},\textbf{j})\in \cR(k,l)$. Then $\ke (\textbf{i}\textbf{j}^{-1},0)\in \cR(k+l,0)$ as $\cR$ is rotation invariant. Since $N$ is normal we have $\cR(k+l,0)= \{\ke(\textbf{f}) \mid a_{\textbf{f}} \in N\}$ and hence $a_{\textbf{i}\textbf{j}^{-1}} \in N$. It follows that  $a_{\textbf{i}} a_{\textbf{j}}^{-1}=a_{\textbf{i}\textbf{j}^{-1}} \in N$. Conversely, let $a_{\textbf{i}} a_{\textbf{j}}^{-1} \in N$. Then $\ke (\textbf{i},\textbf{j}) \in \cR(k,l)$ as it is a rotation of $\ke (\textbf{i}\textbf{j}^{-1})$ and $\cR=\{p\mid p \text{ is a rotation of  } \ke(\textbf{i}) \text{ for some } a_{\textbf{i}} \in N\}$ by assumption.
\item We have $\sqcup=\ke((1,1),0) \in \cR$ since $a_1^2=1\in N$ and $|=\ke((1),(1))\in \cR$ as $a_1a_1^{-1}=1\in N$.
\item For all $p=\ke (\textbf{i},\textbf{j}) \in \cR$ we have 
\[ p^*=\ke (\textbf{j},\textbf{i}) \text{ and } a_{\textbf{j}}a_{\textbf{i}}^{-1}=(a_{\textbf{i}}a_{\textbf{j}}^{-1})^{-1} \in N \]
and hence $\cR$ is closed under involution.
\item Let $p=\ke (\textbf{i},\textbf{j}) \in \cR$ and $q=\ke (\textbf{f},\textbf{g}) \in \cR$ and we consider the connected tensor product $p\ot_{(\textbf{i}\textbf{f},\textbf{j}\textbf{g})} q = \ke (\textbf{i}\textbf{f},\textbf{j}\textbf{g})$. Since 
\[ a_{\textbf{i}\textbf{f}}a_{\textbf{j} \textbf{g}}^{-1} = a_{\textbf{i}}a_{\textbf{f}}a_{\textbf{g}}^{-1} a_{\textbf{j}}^{-1}=a_{\textbf{j}}((a_{\textbf{j}}^{-1}(a_{\textbf{i}}a_{\textbf{j}}^{-1})a_{\textbf{j}})(a_{\textbf{f}}a_{\textbf{g}}^{-1})) a_{\textbf{j}}^{-1} \in N\]
$\cR$ is closed under taking connected tensor products.
\item Let $p=\ke (\textbf{i},\textbf{j}) \in \cR$ and $q=\ke (\textbf{f},\textbf{g}) \in \cR$ be compatible partitions. Then we can assume that $\textbf{j}=\textbf{f}$ and thus we have $qp=\ke(\textbf{i},\textbf{g})$. Since 
\[ a_{\textbf{i}}a_{\textbf{g}}^{-1}=a_{\textbf{i}}a_{\textbf{j}}^{-1}a_{\textbf{j}}a_{\textbf{g}}^{-1}\in N \]
$\cR$ is closed under conditioned composition.
\end{enumerate} 
\ \\
(3) At last, we consider an $\Sym_{\infty}$-invariant normal subgroup $N\unlhd \ZZ_2^{*\infty}$ and define
\[ \cR:= \{p\mid p \text{ is a rotation of  } \ke(\textbf{i}) \text{ for some } a_{\textbf{i}} \in N\}.\]
Then we have $N=F_{\infty}(\cR)$ by Remark \ref{rem::Finfty} and the claim follows by step (2).
\end{proof}

For all $n\in \NN$ we can naturally embed the $n$-fold free product groups $\ZZ_2^{*n}=\langle a_1,\ldots,a_n \rangle \hookrightarrow \langle a_1,a_2,\ldots \rangle =\ZZ_2^{*\infty}$. Theorem \ref{ThMain1} implies directly that for any skew category of partitions words in $F_{\infty}(\cR) \cap \ZZ_2^{*n}$ correspond to the set $\cR_n$ of partitions in $\cR$ with $n$ or less blocks (recall Def. \ref{def::partition}).
\begin{corollary} \label{cor::R_n}
Let $\cR$ be a skew category of partitions and $n\in \NN$. Then $F_n(\cR):=F_{\infty}(\cR) \cap \ZZ_2^{*n} \unlhd \ZZ_2^{*n}$ is an $\Sym_n$-invariant normal subgroup of $\ZZ_2^{*n}$ and we have 
\[ \cR_n = \{p\mid p \text{ is a rotation of  } \ke(\textbf{i}) \text{ for some } a_{\textbf{i}} \in F_n(\cR)\}.\]
\end{corollary}
$\phantom{a}$

\section{Tensor categories of linear maps}
In the previous section we analysed the structure of skew categories of partitions. Our goal is to link skew categories of partitions to the intertwiner spaces of group-theoretical quantum groups. For this purpose we associate linear maps to partitions and show that the linear maps corresponding to a skew category of partitions form a tensor category. In the following we denote $\underline{n}:=\{1,\ldots,n\}$ for all $n\in \NN$.

\begin{definition}
Let $n\in \NN$. For all $p\in P(k,l)$ we define a linear map $\widehat{T}^{(n)}_p \in \Homm((\CC^n)^{\ot k},(\CC^n)^{\ot l})$ via
\begin{align*}
&\widehat{T}^{(n)}_p (e_{i_1} \ot \cdots \ot e_{i_k}) = \sum_{\textbf{j}\in \underline{n}^l}  \widehat{\delta}_p(\textbf{i},\textbf{j}) ~ e_{j_1} \ot \cdots \ot e_{j_l} \\
&\text{with } \widehat{\delta}_p := \mathbb{1}_{\Sym_{\infty}(\ind(p))}
\end{align*}  
for all $\textbf{i}=(i_1,\ldots,i_k) \in \underline{n}^k$.
\end{definition}

Note that we have $\widehat{T}^{(n)}_p=0$ if $p$ has more than $n$ blocks since in this case $\widehat{\delta}_p(\textbf{i},\textbf{j})=0$ for all $(\textbf{i},\textbf{j})\in \underline{n}^k \times \underline{n}^l$. The maps $\widehat{T}^{(n)}_p$ differ from the maps $T^{(n)}_p$ introduced by Banica and Speicher \cite[Def. 1.6,1.7]{BS09} (see also Def. \ref{def::T_p}) in the sense that $\widehat{\delta}_p$ is defined via $\Sym_{\infty}(\ind(p))$ and $\delta_p$ is defined via $\sSym_{\infty}(\ind(p))$ with $\sSym_{\infty}:=\{ \phi: \NN \to \NN \mid |\{n\in \NN \mid \phi(n)\neq n\}|<\infty \}$.

\newpage
The following lemma shows that the operations of skew categories of partitions and $p\mapsto \widehat{T}^{(n)}_p$ behave nicely. 
\begin{lemma} \label{lem::T_p_and_skew_opertions}
Let $n\in \NN,p\in P(k,l)$ and $q \in P(k',l')$. 
\begin{enumerate}[(i)]
\item We have 
\[ (\widehat{T}^{(n)}_p)^* = \widehat{T}^{(n)}_{p^*} .\]
\item We define $L:=\{ p\ot_{(\textbf{if},\textbf{jg})} q \mid (\textbf{i},\textbf{j})\in \Sym_{\infty}(\ind(p)), (\textbf{f},\textbf{g})\in \Sym_{\infty}(\ind(q))\}$. Then we have 
\[\widehat{T}^{(n)}_p \ot \widehat{T}^{(n)}_q = \sum_{r\in L} \widehat{T}^{(n)}_r .\]
\item Let $l=k'$. If $p$ and $q$ are not compatible then $\widehat{T}^{(n)}_q \circ \widehat{T}^{(n)}_p =0$. Otherwise we define $M:=\{ q\cdot_{(\textbf{i},\textbf{h},\textbf{g})} p \mid (\textbf{i},\textbf{h})\in \Sym_{\infty}(\ind(p)), (\textbf{h},\textbf{g})\in \Sym_{\infty}(\ind(q))\}$. Then we have
\[ \widehat{T}^{(n)}_q \circ \widehat{T}^{(n)}_p = \sum_{r\in M} \left( \prod_{z=\text{bl}(r)}^{\text{bl}(r)+l(q,p)-1} n-z \right) \widehat{T}^{(n)}_r \]
where we put $\prod_{z=\text{bl}(r)}^{\text{bl}(r)+l(q,p)-1} n-z =1$ if $l(q,p)=0$.
\end{enumerate}
\end{lemma}

\begin{proof}
An easy calculation of the coefficients of $(\widehat{T}^{(n)}_{p})^*,\widehat{T}^{(n)}_{p} \ot \widehat{T}^{(n)}_{q}$ and $\widehat{T}^{(n)}_{q}\circ \widehat{T}^{(n)}_{p}$ (see \cite[Prop.1.9]{BS09}) shows that it suffices to prove the following for all $(\textbf{i},\textbf{j})\in \NN^k\times \NN^l$ and $(\textbf{f},\textbf{g})\in \NN^{k'}\times \NN^{l'}$.
\begin{enumerate}[(i)]
\item We have $\widehat{\delta}_p (\textbf{i},\textbf{j}) = \widehat{\delta}_{p^*}(\textbf{j},\textbf{i})$.
\item We have $\widehat{\delta}_p (\textbf{i},\textbf{j}) \cdot \widehat{\delta}_q (\textbf{f},\textbf{g}) = \sum_{r\in L} \widehat{\delta}_r (\textbf{if},\textbf{jg})$.
\item Let $l=k'$. Then we have 
\[ \sum_{\textbf{h}\in \underline{n}^l} \widehat{\delta}_p (\textbf{i},\textbf{h}) \cdot \widehat{\delta}_q (\textbf{h},\textbf{g}) =0 \] if $p$ and $q$ are not compatible and 
\[ \sum_{\textbf{h}\in \underline{n}^l} \widehat{\delta}_p (\textbf{i},\textbf{h}) \cdot \widehat{\delta}_q (\textbf{h},\textbf{g}) = \sum_{r\in M} \left( \prod_{z=\text{bl}(r)}^{\text{bl}(r)+l(q,p)-1} n-z \right) \widehat{\delta}_{r}(\textbf{i},\textbf{g})\] 
otherwise.
\end{enumerate}

By definition $\widehat{\delta}_{p}(\textbf{i},\textbf{j})=1$ if and only if $(\textbf{i},\textbf{j})\in \Sym_{\infty}(\ind(p))$. By the definition of involution this is equivalent to $(\textbf{j},\textbf{i})\in \Sym_{\infty}(\ind(p^*))$ which holds if and only if $\widehat{\delta}_{p^*} (\textbf{j},\textbf{i})=1$. This proves (i).

As for (ii), we have $\widehat{\delta}_p (\textbf{i},\textbf{j}) \cdot \widehat{\delta}_q (\textbf{f},\textbf{g})=1$ if and only if $(\textbf{i},\textbf{j})\in \Sym_{\infty}(\ind(p))$ and $(\textbf{f},\textbf{g})\in \Sym_{\infty}(\ind(q))$. By the definition of the connected tensor product and the set $L$ this is equivalent to $\ke(\textbf{if},\textbf{jg}) \in L$. This holds if and only if $\sum_{r\in L} \widehat{\delta}_r (\textbf{if},\textbf{jg}) = \widehat{\delta}_{p\ot_{(\textbf{if},\textbf{jg})} q} (\textbf{if},\textbf{jg}) =1$.

For proving (iii), let $l=k'$ and recall Definition \ref{def::ker_ind}. At first we consider the case that $p$ and $q$ are not compatible. Let $\textbf{h}\in \underline{n}^l$ and $\delta_{\widehat{p}}(\textbf{i},\textbf{h})=1$. This implies $\textbf{h} \in \Sym_{\infty}(\textbf{j}^{(p)})$. By the definition of compatibility we have $\Sym_{\infty}(\textbf{j}^{(p)})\cap \Sym_{\infty}(\textbf{i}^{(q)})=\emptyset$ and hence $\textbf{h} \notin \Sym_{\infty}(\textbf{i}^{(q)})$. It follows that $\delta_{\widehat{q}}(\textbf{h},\textbf{g})=0$. Similarly, one can show that $\delta_{\widehat{q}}(\textbf{h},\textbf{g})=1$ implies $\delta_{\widehat{p}}(\textbf{i},\textbf{h})=0$. Thus we have $\sum_{\textbf{h}\in \underline{n}^l} \widehat{\delta}_p (\textbf{i},\textbf{h}) \cdot \widehat{\delta}_q (\textbf{h},\textbf{g})=0$.\\
So let $p$ and $q$ be compatible. By definition of $M$ there exists a partition $r\in M$ with $\widehat{\delta}_{r}(\textbf{i},\textbf{g})=1$ if and only if the set $\{ \textbf{h}\in \underline{n}^l \mid (\textbf{i},\textbf{h})\in \Sym_{\infty}(\ind(p)), (\textbf{h},\textbf{g})\in \Sym_{\infty}(\ind(q)) \}=\{ \textbf{h}\in \underline{n}^l \mid \widehat{\delta}_p (\textbf{i},\textbf{h}) = \widehat{\delta}_q (\textbf{h},\textbf{g})=1 \}$ is not empty. So we assume that this set is not empty and determine its cardinality. We consider the connected conditioned composition of $p$ and $q$ where we obtain $r$ and a multi-index $\textbf{h}\in \underline{n}^l$ with $\widehat{\delta}_p (\textbf{i},\textbf{h}) = \widehat{\delta}_q (\textbf{h},\textbf{g})=1$. Then there are $\text{bl}(r)$ blocks of $p$ or $q$ which are not loops in the composition and hence the corresponding entries of $\textbf{h}$ are determined by $\textbf{i}$ and $\textbf{g}$. So there are $n-\text{bl}(r)$ numbers left to label the loops of the composition and hence we have $\prod_{z=\text{bl}(r)}^{\text{bl}(r)+l(q,p)-1} n-z$ different choices to do so. Thus the claim follows.
\end{proof}

Now, we can prove that $p\mapsto \widehat{T}^{(n)}_p$ maps skew categories to tensor categories with duals.

\begin{definition}[Tensor category with duals {\cite[Def.1.2,3.5]{BS09}}] \label{def::tensor_cat}
Let $n\in \NN$. We call a collection of vector spaces $\mathcal{H}(k,l) \subseteq \Homm((\CC^n)^{\ot k}, (\CC^n)^{\ot l})$ for all $k,l\in \NN_0$ a \emph{tensor category with duals}, if the following holds.
\begin{enumerate}[(i)]
\item $\mathcal{H}$ is closed under taking tensor products, i.e. $T \in \mathcal{H}(k,l),T' \in \mathcal{H}(k',l')$ implies $T\otimes T' \in \mathcal{H}(k,l)\otimes \mathcal{H}(k',l') \cong \mathcal{H}(k+k',l+l')$.
\item $\mathcal{H}$ is closed under composition, i.e. $T \in \mathcal{H}(k,l)$ and  $T' \in \mathcal{H}(l,l')$ implies $T'T \in\mathcal{H}(k,l')$.
\item $\mathcal{H}$ is closed under involution, i.e. $T \in \mathcal{H}(k,l)$ implies $T^* \in \mathcal{H}(k,l)^*\cong \mathcal{H}(l,k)$.
\item The identity id$:\CC^n \to \CC^n,~x\mapsto x$ is in $\mathcal{H} (1,1)$.
\item Let $e_1,\ldots ,e_n$ be the standard basis of $\CC^n$. Then $\zeta: \CC \to (\CC^n)^{\ot 2}, 1 \mapsto \sum_{i=1}^n e_i \ot e_i$ is in $\mathcal{H} (0,2)$.
\end{enumerate}
\end{definition}

Recall, that for a subset $D \subseteq P$ we denote the set of all partitions in $D$ with at most $n$ blocks by $D_n$ (see Def. \ref{def::partition}).

\begin{lemma} \label{lem::lin_independent}
For all $k,l\in \NN_0$ the elements $\{\widehat{T}_p^{(n)} \mid p\in P_n(k,l) \}$ are linearly independent.
\end{lemma}

\begin{proof}
Let $m\in \NN$, $a_x\in \CC$ and $p_x\in P_n(k,l)$ for all $x\in \underline{m}$ such that $\sum_{x=1}^m a_x \widehat{T}_{p_x}^{(n)} =0$.
Then we have
\[ 0=\sum_{x=1}^m a_x \widehat{T}_{p_x}^{(n)} (e_{i_1}\ot \ldots \ot e_{i_k}) 
= \sum_{x=1}^m \sum_{\textbf{j}\in \underline{n}^l} a_x \widehat{\delta}_{p_x} (\textbf{i},\textbf{j}) ~e_{j_1}\ot \ldots \ot e_{j_l} \text{ for all } \textbf{i} \in \underline{n}^k \]
and hence
\[ 0= \sum_{x=1}^m a_x \widehat{\delta}_{p_x} (\textbf{i},\textbf{j})  \text{ for all } (\textbf{i},\textbf{j}) \in \underline{n}^k \times \underline{n}^l.\]
Let $y\in \underline{m}$ and we evaluate this equation in $\ind (p_y) \in \underline{n}^k \times \underline{n}^l$. Note that for all $x\in \underline{m}$ we have $\widehat{\delta}_{p_x} (\ind(p_y))=1$ if and only if $\ind(p_y) \in \Sym_{\infty}(\ind(p_x))$. As $p_x$ has at most $n$ blocks for all $x\in \underline{m}$ we can encode $p_x$ uniquely by $\ind(p_x)$ and hence $\ind(p_y) \in \Sym_{\infty}(\ind(p_x))$ if and only $x=y$. This implies
\[ 0 = \sum_{x=1}^m a_x \widehat{\delta}_{p_x} (\ind (p_y)) = a_y \]
and it follows that the elements $\{\widehat{T}_p^{(n)} \mid p\in P_n(k,l) \}$ are linearly independent.
\end{proof}

\begin{theorem} \label{ThMain2}
Let $\cR\subseteq P$. Then $\cR$ is a skew category of partitions if and only if $\lspan \{\widehat{T}_p^{(n)} \mid p\in \cR(k,l)\}, k,l\in \NN_0$ is a tensor category with duals for all $n\in \NN$.  
\end{theorem}

\begin{proof}
If $\cR$ is a skew category of partitions, Lemma \ref{lem::T_p_and_skew_opertions} and Lemma \ref{lem::skew_cat} imply that $\lspan \{\widehat{T}_p^{(n)} \mid p\in \cR(k,l)\},k,l\in \NN_0$ is a tensor category with duals for all $n\in \NN$. \\
To prove the converse note that $\lspan \{\widehat{T}_p^{(n)} \mid p\in \cR(k,l)\}=\lspan \{\widehat{T}_p^{(n)} \mid p\in \cR_n(k,l) \}$ for all $k,l\in \NN_0$ and the elements $\{\widehat{T}_p^{(n)} \mid p\in \cR_n(k,l) \}$ are lineraly independent by Lemma \ref{lem::lin_independent}. Thus if $\lspan \{\widehat{T}_p^{(n)} \mid p\in \cR(k,l)\}$ is a tensor category with duals for all $n\in \NN$ it follows from Lemma \ref{lem::T_p_and_skew_opertions} that $\cR$ is a skew category of partitions.
\end{proof}
\medskip

\section{The intertwiner spaces of group-theoretical quantum groups}
In this section we will prove that the tensor categories with duals induced by skew categories of partitions model the intertwiner spaces of most of the group-theoretical quantum groups. Therefore, we have to recall some basics on compact matrix quantum groups, easy quantum groups and group-theoretical quantum groups. 

\subsection{Compact matrix quantum groups and easy quantum groups}
For any compact group $G$ the function space $C(G)$ is a commutative C*-algebra with a comultiplication which fulfils some dualised group properties. Woronowicz generalised this concept in a non-commutative setting by defining compact quantum groups as C*-algebras with comultiplication fulfilling the dualised group properties. In fact a compact quantum group arises from a compact group if and only if its C*-algebra is commutative. In 1987, Woronowicz defined compact matrix quantum groups, a subclass of compact quantum groups, as follows.

\begin{definition}[Compact matrix quantum group {\cite{Wo87}}] \label{def::CMQG}
$\phantom{a}$\\
A \emph{compact matrix quantum group (CMQG)} consists of a C*-algebra $A$, a matrix $u=(u_{ij}) \in A^{n\times n}$, called \emph{fundamental corepresentation}, and a *-homomorphism $\Delta :A \to A \ot A$, called \emph{comultiplication}, such that the following holds.
\begin{enumerate}[1.)]
\item The elements $\{ u_{ij} \mid 1\leq i,j \leq n \}$ generate $A$ in the sense that the generated *-algebra is dense in $A$.
\item The matrix $u=(u_{ij})$ is unitary and its transpose $u^t=(u_{ji})$ is invertible.
\item We have $\Delta(u_{ij})=\sum_{k=1}^n u_{ik} \ot u_{kj}$ for all $i,j\in \underline{n}$.
\end{enumerate}
We write $(A,u,n)$ for a compact matrix quantum group $A$ with fundamental corepresentation $u=(u_{ij}) \in A^{n\times n}$. (The common notation is $(A,u)$ as the size $n$ is usually implicit but this is not always the case in this article.)  
\end{definition}

\begin{example}
Let $n\in \NN$. Wang defined the \emph{free symmetric quantum group} $\Sym_n^+$ (cf. \cite{Wa98}) and the \emph{free orthogonal quantum group} $O_n^+$ (cf. \cite{Wa95}) via
\begin{align*}
C(\Sym_n^+) := C^*(u_{ij}, 1\leq i,j\leq n\mid &u_{ij}^*=u_{ij}=u_{ij}^2,~\sum_{k=1}^n u_{ik} = \sum_{k=1}^n u_{kj} = 1,\\
 &u_{ij}u_{ik}=u_{ji}u_{ki}=0 \text{ for } j\neq k),\\
C(O_n^+):= C^*(u_{ij}, 1\leq i,j\leq n \mid &u_{ij}^*=u_{ij},~u^tu=uu^t=1).
\end{align*}
We have 
\begin{align*}
&C(\Sym_n) \cong C(\Sym_n^+)/\{u_{ij}u_{kl}=u_{kl}u_{ij}\} \text{ and } \\
&C(O_n) \cong C(O_n^+)/\{u_{ij}u_{kl}=u_{kl}u_{ij}\}.
\end{align*}
\end{example}

We refer for instance to \cite{RW15} for the next two definitions.

\begin{definition}[Compact matrix quantum subgroup] \label{def::quantum_subgroup}
Let $(A,u,n)$ and $(B,v,m)$ be CMQGs. Then $(A,u,n)$ is called a \emph{compact matrix quantum subgroup} of $(B,v,m)$ if $n=m$ and if there exists a surjective *-homomorphism  preserving the fundamental corepresentation $A \overset{\sim}{\twoheadrightarrow} B$.  \\
\end{definition}

In the following we will only consider CMQGs $(A,u,n)$ in their maximal versions which have $C(S_n)$ as compact matrix quantum subgroup and are compact matrix quantum subgroups of  $O_n^+$ (see \cite[§ 2.2]{RW15} for more details). Thus in the following any fundamental representation of a CMQG will be orthogonal.

In 1988, Woronowicz proved a Tannaka-Krein type result for CMQGs showing that any CMQG can be fully recovered from its intertwiner spaces.

\begin{definition}[Intertwiner spaces]
Let $(A,u,n)$ be a CMQG. Then the \emph{intertwiner spaces of $A$} for $k,l\in \NN_0$ are defined as
\[ \Homm_{A}(k,l) := \{ T:(\CC^n)^{\ot k} \to (\CC^n)^{\ot l} \text{ linear} \mid Tu^{\ot k} = u^{\ot l} T \} ,\]
where we think of $u^{\ot k}$ and $u^{\ot l}$ as endomorphism of $(\CC^n)^{\ot k} \ot A$ and $(\CC^n)^{\ot l} \ot A$, respectively, and where we view $T$ as a morphism from $(\CC^n)^{\ot k} \ot A$ to $(\CC^n)^{\ot l} \ot A$ with $T(e_{i_1}\ot \cdots \ot e_{i_k} \ot \hat{u}) = T(e_{i_1}\ot \cdots \ot e_{i_k}) \ot \hat{u}$ for all $i_1,\ldots ,i_k \in \underline{n}$ and $\hat{u} \in A$.
\end{definition}

\begin{theorem}[Tannaka-Krein {\cite{Wo88}}] \label{thm::tannaka_krein}
The following construction induces an \linebreak inclusion-inverting one-to-one correspondence between CMQGs and tensor categories with duals:\\
(1) To a CMQG $(A,u,n)$ we associate its intertwiner spaces $\Homm_{A}(k,l),k,l\in \NN_0$ which from a tensor category with duals. \\
(2) To a tensor category with duals $\mathcal{H}(k,l) \subseteq \Homm((\CC^n)^{\ot k}, (\CC^n)^{\ot l}), k,l\in \NN_0$ we associate the CMQG 
\[C^*(u_{ij}, 1\leq i,j\leq n \mid u_{ij}^*=u_{ij},~Tu^{\ot k} = u^{\ot l} T \text{ for all } k,l\in \NN_0,T\in \mathcal{H}(k,l)).\]
\end{theorem}

\begin{remark} \label{rem::tannaka_krein}
By an easy calculation we see that 
\begin{align*}
&C^*(u_{ij}, 1\leq i,j\leq n \mid u_{ij}^*=u_{ij},~Tu^{\ot k} = u^{\ot l} T \text{ for all } k,l\in \NN_0,T\in \mathcal{H}(k,l))\\
&= C^*(u_{ij}, 1\leq i,j\leq n \mid u_{ij}^*=u_{ij},~Tu^{\ot k} = T \text{ for all } k\in \NN_0,T\in \mathcal{H}(k,0))
\end{align*}
for any tensor category with duals $\mathcal{H}(k,l) \subseteq \Homm((\CC^n)^{\ot k}, (\CC^n)^{\ot l})$, $k,l\in \NN_0$.
\end{remark}

Based on this result, Banica and Speicher defined (orthogonal) easy quantum groups in 2009 as follows. 

\begin{definition}[Strongly symmetric semigroup]
The \emph{strongly symmetric semigroup} $\sSym_n$ is the semigroup of all maps $\{\phi:\underline{n} \to \underline{n}\}$ and we define the \emph{strongly symmetric semigroup (on countably, but infinitely many points)} $\sSym_{\infty}$ as
\[ \sSym_{\infty}:= \{ \phi: \NN \to \NN \mid |\{n\in \NN \mid \phi(n)\neq n\}|<\infty \} .\] 
\end{definition}

Note that similarly to the symmetric group the strongly symmetric semigroup $\sSym_{\infty}$, respectively $\sSym_n$, acts componentwise on multi-indices in $\NN^k \times \NN^l$, respectively $\underline{n}^k \times \underline{n}^l$, for all $k,l\in \NN_0$.

\begin{definition}{\cite[Def. 1.6,1.7]{BS09}} \label{def::T_p}
For all $n\in \NN$ and $p\in P(k,l)$ we define a linear map $T^{(n)}_p \in \Hom((\CC^n)^{\ot k},(\CC^n)^{\ot l})$ via
\begin{align*}
&T^{(n)}_p (e_{i_1} \ot \cdots \ot e_{i_k}) = \sum_{\textbf{j}\in \underline{n}^l}  \delta_p(\textbf{i},\textbf{j}) ~ e_{j_1} \ot \cdots \ot e_{j_l} \\
&\text{with } \delta_p := \mathbb{1}_{\sSym_n(\ind(p))} 
\end{align*}  
for all $\textbf{i}=(i_1,\ldots,i_k) \in \underline{n}^k$.
\end{definition} 

Banica and Speicher showed that $\lspan \{T_p^{(n)} \mid p\in \cC(k,l)\},k,l\in \NN_0$ is a tensor category with duals for all categories of partitions $\cC$ which leads to the following definition.

\begin{definition}[(Orthogonal) easy quantum group {\cite[Def. 3.5]{BS09}}]
A compact matrix quantum group $(A,u,n)$ is called \emph{(orthogonal) easy quantum group} if there exists a category of partitions $\cC \subseteq P$ such that for all $k,l\in \NN_0$
\[ \Homm_{A} (k,l) = \text{span} \{ T^{(n)}_p \mid p\in \cC (k,l) \} ~.\]
\end{definition}

\subsection{Group-theoretical quantum groups}
In 2015, Raum and Weber introduced group-theoretical quantum groups. We roughly summarise their results.

\begin{definition}[Group-theoretical quantum group {\cite{RW15}}]
A compact matrix quantum group $(A,u,n)$ is called \emph{group-theoretical} if $u_{ij}^2$ are central projections in $A$ for all $i,j\in \underline{n}$.
\end{definition}

\begin{remark}
Let $(A,u,n)$ be a compact matrix quantum group and $p:=\primarypart$. Then $(A,u,n)$ is a group-theoretical quantum group if and only if $T^{(n)}_p \in \Homm_A(3,3)$. This holds if and only if $\widehat{T}^{(n)}_p \in \Homm_A(3,3)$.
\end{remark}

\begin{proof}
It is an easy calculation to show that $(A,u,n)$ is group-theoretical if and only if $T^{(n)}_p \in \Homm_A(3,3)$ (see \cite[Prop.2.5]{RW15}). \\
Consider the partition $e_{(3,3)}=\ke((1,1,1),(1,1,1)) \in P(3,3)$. 
It is easy to check  that $\widehat{T}^{(n)}_p=T^{(n)}_p-\widehat{T}^{(n)}_{e_{(3,3)}}$ and $\widehat{T}^{(n)}_{e_{(3,3)}}=T^{(n)}_{e_{(3,3)}}$ (see also Lem. \ref{lemma::hatp_relation}). Since $e_{(3,3)}=pp$, we have $T^{(n)}_{e_{(3,3)}}= T^{(n)}_p \circ T^{(n)}_p$ and thus $T^{(n)}_p \in \Homm_A(3,3)$ implies $\widehat{T}^{(n)}_p \in \Homm_A(3,3)$.\\
Now, let $\widehat{T}^{(n)}_p \in \Homm_A(3,3)$. It suffices to show that $e_{(3,3)} \in \langle p \rangle_{skew}$ since then Lemma \ref{lem::T_p_and_skew_opertions} implies $\widehat{T}^{(n)}_{e_{(3,3)}} \in \Homm_A(3,3)$ and hence $T^{(n)}_p \in \Homm_A(3,3)$. As $p \in \langle p \rangle_{skew}$ and $\langle p \rangle_{skew}$ is closed under rotation we have

\begin{minipage}[t]{.38\linewidth}
\vspace{0pt}
\begin{tikzpicture}
\coordinate [label=right:{$=pp^*\in \langle p \rangle_{skew}$ and}](O2) at (1.2,0);
\coordinate (A1) at (0,-0.5);
\coordinate (A2) at (0.5,-0.5);
\coordinate (A3) at (1,-0.5);
\coordinate (A4) at (0,0.5);
\coordinate (A5) at (0.5,0.5);
\coordinate (A6) at (1,0.5);

\coordinate (C1) at (0.5,0.2);
\coordinate (C2) at (0.5,-0.2);
\coordinate (C3) at (1,0.2);
\coordinate (C4) at (1,-0.2);
\coordinate (C5) at (0.75,0.2);
\coordinate (C6) at (0.75,-0.2);

\fill (A1) circle (2.5pt);
\fill (A2) circle (2.5pt);
\fill (A3) circle (2.5pt);
\fill (A4) circle (2.5pt);
\fill (A5) circle (2.5pt);
\fill (A6) circle (2.5pt);

\draw (A1) -- (A4);
\draw (A5) -- (C1) -- (C3) -- (A6);
\draw (A2) -- (C2) -- (C4) -- (A3);
\draw (C5) -- (C6);
\end{tikzpicture}
\end{minipage} 
\begin{minipage}[t]{.45\linewidth}
\vspace{0pt}
\begin{tikzpicture}
\coordinate [label=left:{$q:=$}](O) at (-0.2,0);
\coordinate [label=right:{$\in \langle p \rangle_{skew}.$}](O2) at (1.2,0);
\coordinate (A1) at (0,-0.5);
\coordinate (A2) at (0.5,-0.5);
\coordinate (A3) at (1,-0.5);
\coordinate (A4) at (0,0.5);
\coordinate (A5) at (0.5,0.5);
\coordinate (A6) at (1,0.5);

\coordinate (C1) at (0,0.2);
\coordinate (C2) at (0.5,0.2);
\coordinate (C3) at (1,0.2);
\coordinate (C4) at (0,-0.2);
\coordinate (C5) at (0.5,-0.2);

\fill (A1) circle (2.5pt);
\fill (A2) circle (2.5pt);
\fill (A3) circle (2.5pt);
\fill (A4) circle (2.5pt);
\fill (A5) circle (2.5pt);
\fill (A6) circle (2.5pt);

\draw (A4) -- (C1) -- (C3) -- (A6);
\draw (A5) -- (C2);
\draw (A6) -- (A3);
\draw (A1) -- (C4) -- (C5) -- (A2);
\end{tikzpicture}
\end{minipage}
\ \\
\ \\
Since $\langle p \rangle_{skew}$ is closed under conditioned composition it follows that $e_{(3,3)}=qq^* \in \langle p \rangle_{skew}$.
\end{proof} 

\begin{lemma} \label{lem::u_iju_ik=0}
Let $(A,u,n)$ be a group-theoretical quantum group. Then we have 
\[u_{ij}^3=u_{ij} \text{ and } u_{ij}u_{ik}=0=u_{ji}u_{ki} \text{ for all } i,j,k\in \underline{n} \text{ with } j\neq k. \]
Hence for all multi-indices $(i_1,\ldots,i_k),(j_1,\ldots,j_k) \in \underline{n}^k$ with $i_x = i_y$ and $j_x\neq j_y$ for some $x,y\in \underline{k}$ we have $u_{i_1j_1}\cdot \ldots \cdot u_{i_kj_k}=0$. In particular, we have $u_{i_1j_1}\cdot \ldots \cdot u_{i_kj_k}=0$ if $\textbf{i}\in \ind(p)$ and $\textbf{j}\notin \ind(p)$ for some partition $p\in P(k,0)$. 
\end{lemma}

\begin{proof}
Since the elements $q_{ij}:=u_{ij}^2$ are projections with $\sum_{j=1}^n q_{ij} = 1$ we have $q_{ij}q_{ik}=0$ and $q_{ji}q_{ki}=0$ for $j\neq k$. Hence for $j\neq k$ it follows that $(u_{ij}u_{ik})(u_{ij}u_{ik})^* =q_{ij}q_{ik}=0$. Thus $\|u_{ij}u_{ik}\|^2=0$ and we have $u_{ij}u_{ik}=0$ for $j\neq k$. Analogously, we obtain $u_{ji}u_{ki}=0$. This implies 
\[ u_{ij} = u_{ij} \sum_{k=1}^n u_{ik}^2 = u_{ij}^3.\]
Using the results above the second claim is easy to check.
\end{proof}

Raum and Weber showed that group-theoretical quantum groups have a presentation as a semi-direct product quantum group (for more details see \cite[§ 2.5]{RW15}).

\begin{theorem}{\cite[Thm. 3.1]{RW15}} \label{thm::raumweber_thm3.1.}
Let $n\in \NN$ and let $N\unlhd \ZZ_2^{*n}$ be an $\Sym_n$-invariant normal subgroup. Then the semi-direct product quantum group given by $C^*(\ZZ_2^{*n}/N)\Join C(\Sym_n)$ is a group-theoretical quantum group.\\
Vice versa, let $(A,u,n)$ be a group-theoretical quantum group. Then there exists an $\Sym_n$-invariant normal subgroup $N\unlhd \ZZ_2^{*n}$ such that $A \cong C^*(\ZZ_2^{*n}/N) \Join C(\Sym_n)$.
\end{theorem}

\begin{remark} \label{rem::thm3.1.}
The proof of this theorem by Raum and Weber also shows the following. Let $n\in \NN$ and let $N\unlhd \ZZ_2^{*n}$ be an $\Sym_n$-invariant normal subgroup. We define $A:=C^*(\ZZ_2^{*n}/N) \Join C(\Sym_n)$.
Then we have 
\begin{align*}
A \cong C^*(u_{ij}, 1\leq i,j\leq n\mid &u_{ij}^*=u_{ij}, u_{ij}^2 \text{ central projections},\\
&\sum_{\textbf{f}=(f_1,\ldots,f_k) \in \underline{n}^k} u_{f_1i_1}\cdot \ldots \cdot u_{f_ki_k}=1 \\
&\text{ for all } k\in \NN_0, \textbf{i}\in \underline{n}^k, a_{\textbf{i}}\in N ) .
\end{align*}
\end{remark} 

Raum and Weber also classified the easy case. Note that similarly to the symmetric group the strongly symmetric semigroup $\sSym_{\infty}$, respectively $\sSym_n$, acts on words in $\ZZ_2^{\infty}$, receptively $\ZZ_2^n$, via $\phi(a_i) = a_{\phi(i)}$ (see Def. \ref{def::Finfty}).

\begin{theorem}{\cite[Thm. 4.4,4.5]{RW15}} \label{thm::raumweber_thm4.4.}
Let $(A,u,n)$ be a group-theo\-retical quantum group with corresponding normal subgroup $N$. Then $(A,u,n)$ is easy if and only if $N$ is $\sSym_n$-invariant. In this case, the corresponding category of partitions is given by
\[ \cC = \langle \ke(\textbf{i}) \mid a_{\textbf{i}} \in N \rangle .\]
\end{theorem}

\begin{remark}
The statement of Theorem 4.4 in \cite{RW15} is correct but the proof contains a slight error. The category of partitions $\cC$ they constructed is not closed under taking tensor products. But this can be easily fixed (see the appendix of our article) and the set they constructed was still a generating set for the category of partitions we are looking for. 
\end{remark}

\subsection{The intertwiner spaces of group-theoretical \\
quantum groups}
We will now determine the intertwiner spaces of group-theoretical quantum groups in general. We begin with two preliminary lemmas.

\begin{lemma} \label{lemma::hatp_relation}
Let $(A,u,n)$ be a group-theoretical quantum group and $p\in P(k,0)$. Then $\widehat{T}^{(n)}_p \in \Homm_{A}(k,0)$ if and only if $p$ has more than $n$ blocks or
\[ \sum_{\textbf{f}=(f_1,\ldots,f_k) \in \underline{n}^k} u_{f_1i_1}\cdot \ldots \cdot u_{f_ki_k} = 1 \text{ for all } \textbf{i}=(i_1,\ldots,i_k) \in \Sym_n(\ind(p)) .\]
\end{lemma}

\begin{proof}
If $p$ has more than $n$ blocks it follows that $\widehat{T}^{(n)}_p=0 \in \Homm_{A}(k,0)$. So assume that $p$ has $n$ or less blocks.
It is easy to check that $\widehat{T}^{(n)}_p \in \Homm_{A}(k,0)$ if and only if 
\[ \sum_{\textbf{f}=(f_1,\ldots,f_k) \in \underline{n}^k} \widehat{\delta}_p(\textbf{f},0) u_{f_1i_1}\cdot \ldots \cdot u_{f_ki_k} = \widehat{\delta}_p(\textbf{i},0) \text{ for all } \textbf{i}=(i_1,\ldots,i_k) \in \underline{n}^k.\]
For $\textbf{i}=(i_1,\ldots,i_k)\in \Sym_n(\ind(p))$ we have $\widehat{\delta}_p(\textbf{i},0)=1$ and
\[ \sum_{\textbf{f} \in \underline{n}^k} \widehat{\delta}_p(\textbf{f},0)~ u_{f_1i_1}\cdot \ldots \cdot u_{f_ki_k} 
=\sum_{\textbf{f} \in \Sym_n(\ind (p))} u_{f_1i_1}\cdot \ldots \cdot u_{f_ki_k} 
\overset{\ref{lem::u_iju_ik=0}}{=}\sum_{\textbf{f} \in \underline{n}^k} u_{f_1i_1}\cdot \ldots \cdot u_{f_ki_k}.\]
For $\textbf{i}=(i_1,\ldots,i_k)\notin \Sym_n(\ind(p))$ we have $\widehat{\delta}_p(\textbf{i},0)=0$ and
\[ \sum_{\textbf{f} \in \underline{n}^k} \widehat{\delta}_p(\textbf{f},0)~ u_{f_1i_1}\cdot \ldots \cdot u_{f_ki_k} 
=\sum_{\textbf{f} \in \Sym_n(\ind (p))} u_{f_1i_1}\cdot \ldots \cdot u_{f_ki_k} 
\overset{\ref{lem::u_iju_ik=0}}{=}0.\]
\end{proof}

The next lemma describes explicitly the skew categories that we will use in the following. For this purpose we have to lift normal subgroups $N\unlhd \ZZ_2^{*n}$ to normal subgroups in $\ZZ_2^{*\infty}$.
\begin{definition} \label{def::Ninfty}
For any $\Sym_n$-invariant normal subgroup $N\unlhd \ZZ_2^{*n}$ we define $\Sym_{\infty}(N):= \{ \sigma (g) \mid \sigma \in \Sym_{\infty}, g\in N \}$ and denote by $N_{\infty}:=\langle \langle \Sym_{\infty}(N) \rangle \rangle_{\ZZ_2^{*\infty}}$ the closure of $\Sym_{\infty}(N)$ as normal subgroup in $\ZZ_2^{*\infty}$.
\end{definition}

\begin{lemma} \label{cor::of_main_3} 
Let $N\unlhd \ZZ_2^{*n}$ be an $\Sym_n$-invariant normal subgroup. Then we have
\begin{align*}
\cR:=& \langle \ke(\textbf{i}) \mid a_{i_1}\cdot \ldots \cdot a_{i_k} \in N \rangle_{skew} \\
=& \{p\mid p \text{ is a rotation of  } \ke(\textbf{i}) \text{ for some } a_{\textbf{i}} \in N_{\infty}\}.
\end{align*}
\end{lemma}

\begin{proof}
Since $N_{\infty}$ is an $\Sym_{\infty}$-invariant normal subgroup of $\ZZ_2^{*\infty}$ the set
\[\cR' = \{p\mid p \text{ is a rotation of  } \ke(\textbf{i}) \text{ for some } a_{\textbf{i}} \in N_{\infty}\} \]
is a skew category of partitions with $F_{\infty}(\cR')=N_{\infty}$ by Remark \ref{rem::Finfty} and Theorem \ref{ThMain1}. Since $\{\ke(\textbf{i}) \mid a_{\textbf{i}} \in N\} \subseteq \cR'$ we have $\cR \subseteq \cR'$. \\
By the definition of $\cR$ we have $N\subseteq F_{\infty}(\cR)$. Again by Theorem \ref{ThMain1} $F_{\infty}(\cR)$ is an $\Sym_{\infty}$-invariant normal subgroup of $\ZZ_2^{*\infty}$ and thus we have $N_{\infty}\subseteq F_{\infty}(\cR)$. This implies $N_{\infty}=F_{\infty}(\cR') \subseteq F_{\infty}(\cR)$ and hence we have $\cR'\subseteq \cR$.
\end{proof}

Recall that any group-theoretical quantum group is a semi-direct product quantum group $C^*(\ZZ_2^{*n}/N) \Join C(\Sym_n)$ for some $\Sym_n$-invariant normal subgroup $N\unlhd \ZZ_2^{*n}$, see \ref{thm::raumweber_thm3.1.}.
\begin{theorem} \label{ThMain3}
Let $(A,u,n)$ be a compact matrix quantum group. Then $(A,u,n)$ is group-theoretical with $N_{\infty} \cap \ZZ_2^{*n} = N$ if and only if there exists a skew category of partitions $\cR$ such that 
\[ \Homm_A(k,l)=\lspan \{\widehat{T}_p^{(n)} \mid p\in \cR(k,l)\} \] 
for all $k,l\in \NN_0$. In this case, we have 
\[ \cR =\langle \ke(\textbf{i}) \mid a_{i_1}\cdot \ldots \cdot a_{i_k} \in F_n(\cR) \rangle_{skew} \]
and 
\[ A \cong C^*(\ZZ_2^{*n}/F_n(\cR)) \Join C(\Sym_n) \]
with $F_n(\cR)=F_{\infty}(\cR)\cap \ZZ_2^{*n}$.
\end{theorem}

\begin{proof}
Let $(A,u,n)$ be a group-theoretical quantum group with $A \cong C^*(\ZZ_2^{*n}/N) \Join C(\Sym_n)$ and $N_{\infty} \cap \ZZ_2^{*n} = N$. We define $\cR =\langle \ke(\textbf{i}) \mid a_{i_1}\cdot \ldots \cdot a_{i_k} \in N \rangle_{skew}$. Then we have $F_{\infty}(\cR)=N_{\infty}$ by Lemma \ref{cor::of_main_3} and hence it follows that $F_n(\cR)=F_{\infty}(\cR) \cap \ZZ_2^{*n}=N_{\infty} \cap \ZZ_2^{*n} = N$. Moreover, the vector spaces $\lspan \{\widehat{T}_p^{(n)} \mid p\in \cR(k,l)\}$ form a tensor category with duals containing $\widehat{T}^{(n)}_p$ for $p=\primarypart$  by Theorem \ref{ThMain2}. Thus by the Tannaka-Krein duality (Theorem \ref{thm::tannaka_krein}), there exists a group-theoretical quantum group $(B,v,n)$ such that $\Homm_B(k,l)=\lspan \{\widehat{T}_p^{(n)} \mid p\in \cR(k,l)\}$ for all $k,l\in \NN_0$ and by Remark \ref{rem::tannaka_krein} we have
\begin{align*}
B \cong C^*(v_{ij}, 1\leq i,j\leq n \mid &v_{ij}^*=v_{ij},~\widehat{T}_p^{(n)} v^{\ot k} = \widehat{T}_p^{(n)} \\
&\text{ for all } k\in \NN_0,p\in \cR(k,0)).
\end{align*}
It follows from Lemma \ref{lemma::hatp_relation} that
\begin{align*}
B \cong C^*(v_{ij}, 1\leq i,j\leq n\mid &v_{ij}^*=v_{ij}, v_{ij}^2 \text{ central projections}, \\
&\sum_{\textbf{f}=(f_1,\ldots,f_k) \in \underline{n}^k} v_{f_1i_1}\cdot \ldots \cdot v_{f_ki_k} = 1 \\
&\text{for all } k\in \NN_0,p\in \cR_n(k,0), \textbf{i} \in \Sym_n(\ind(p)) )
\end{align*} 
and by the definition of $\cR$ we have 
\begin{align*}
B \cong C^*(v_{ij}, 1\leq i,j\leq n\mid &v_{ij}^*=v_{ij}, v_{ij}^2 \text{ central projections}, \\
&\sum_{\textbf{f}=(f_1,\ldots,f_k) \in \underline{n}^k} v_{f_1i_1}\cdot \ldots \cdot v_{f_ki_k} = 1 \\
&\text{for all } k\in \NN_0, \textbf{i} \in \underline{n}^k, a_{\textbf{i}}\in N ).
\end{align*} 
By Remark \ref{rem::thm3.1.} we have $A\cong B$. \\
Now, let $(A,u,n)$ just be a compact matrix quantum group such that there exists a skew category of partitions $\cR$ with $\Homm_A(k,l)=\lspan \{\widehat{T}_p^{(n)} \mid p\in \cR(k,l)\}$ for all $k,l\in \NN_0$. We consider $\cR':= \langle \cR_n \rangle_{skew}$, the skew category generated by all partitions in $\cR$ with at most $n$ blocks. Then $F_\infty (R')$ is generated by elements in $\ZZ_2^{*n}$ and hence $F_\infty (R')=(F_\infty (R')\cap \ZZ_2^{*n})_\infty$. We set $N:=F_\infty (R')\cap \ZZ_2^{*n} \unlhd \ZZ_2^{*n}$ and have $N_{\infty} \cap \ZZ_2^{*n} = N$. Hence we can apply the first part of the proof to the corresponding group-theoretical quantum group $(B,v,n)$ with $B\cong C^*(\ZZ_2^{*n}/N) \Join C(\Sym_n)$ and obtain
$\Homm_B(k,l)=\lspan \{\widehat{T}_p^{(n)} \mid p\in \cR'(k,l)\}$ for all $k,l\in \NN_0$. Since $\cR_n=(\cR')_n$  by construction and $\widehat{T}^{(n)}_p=0$ if $p$ has more than $n$ blocks, we have $\Homm_A(k,l)=\Homm_B(k,l)$ for all $k,l\in \NN_0$. It follows that $A\cong B$ by the Tannaka-Krein duality, which completes the proof.
\end{proof}

Although Lemma \ref{cor::of_main_3} and Theorem \ref{ThMain3} showed that we have to lift the normal subgroup $N\unlhd \ZZ_2^{*n}$ to $\ZZ_2^{\infty}$ to describe the whole skew category of partitions $\cR$, the following corollary shows that $N$ still suffices to describe the whole intertwiner space.
\begin{corollary}
Let $N\unlhd \ZZ_2^{*n}$ be an $\Sym_n$-invariant normal subgroup with $N_{\infty} \cap \ZZ_2^{*n} = N$, $A:=C^*(\ZZ_2^{*n}/N) \Join C(\Sym_n)$ and $\cR =\langle \ke(\textbf{i}) \mid a_{i_1}\cdot \ldots \cdot a_{i_k} \in N \rangle_{skew}$ as in the previous theorem. Then we have
\begin{align*}
&\Homm_A(k,l)=\lspan \{\widehat{T}_p^{(n)} \mid p\in \cR_n\} \text{ and}\\
&\cR_n = \{p\mid p \text{ is a rotation of  } \ke(\textbf{i}) \text{ for some } a_{\textbf{i}} \in N\}.
\end{align*}  
\end{corollary}

\begin{proof}
This follows directly from the fact that $\widehat{T}_p^{(n)}=0$ for all $p\in \cR \backslash \cR_n$ and Corollary \ref{cor::R_n}.
\end{proof}

\begin{remark} \label{rem::gen_case}
We have seen that the structure of (the intertwiner spaces of) group-theoretical quantum groups whose corresponding normal subgroup $N\unlhd \ZZ_2^{*n}$ satisfies the condition $N_{\infty} \cap \ZZ_2^{*n} = N$ can be pictured in skew categories of partitions. In particular, any skew category of partitions $\cR$ gives rise to a series of such group-theoretical quantum groups $G_n$ with
\[ C(G_n) = C^*(\ZZ_2^{*n}/(F_\infty(\cR)\cap \ZZ_2^{*n})) \Join C(\Sym_n) .\]
Normal subgroups that do not satisfy the above condition are those which can not be obtained as a restriction of a normal subgroup of $\ZZ_2^{*\infty}$ and hence which depend explicitly on the 'dimension' $n$. Then the associated group-theoretical quantum groups do not appear in a series of quantum groups as above and are in some sense even further away from the easy case.\\ 
\ \\
However, we can still describe the intertwiner spaces of group-theoretical quantum groups in general. For an arbitrary $\Sym_n$-invariant normal subgroup $N\unlhd \ZZ_2^{*n}$ the intertwiner spaces of the corresponding group-theoretical quantum group $A:=C^*(\ZZ_2^{*n}/N) \Join C(\Sym_n)$ are given by 
\begin{align*}
&\Homm_A(k,l)=\lspan \{\widehat{T}_p^{(n)} \mid p\in \mathcal{S} \} \text{ with}\\
&\mathcal{S} = \{p\mid p \text{ is a rotation of  } \ke(\textbf{i}) \text{ for some } a_{\textbf{i}} \in N\}.
\end{align*}
The proof of this statement is rather long and technical, which is why we only give the idea here. The strategy is to adapt the arguments of this article to partitions with at most $n$ blocks. Similar to Theorem \ref{ThMain1}, one can check that $\mathcal{S}$ is closed under those modified operations (involution, connected tensor products, conditioned composition) which yields partitions with at most $n$ blocks. Then Theorem \ref{ThMain2} together with the fact that $\widehat{T}^{(n)}_p=0$ if $p$ has more than $n$ blocks, implies that $\lspan \{\widehat{T}_p^{(n)} \mid p\in \mathcal{S} \}$ is a tensor category with duals. Using this property, the statement follows analogous to the proof of Theorem \ref{ThMain3}.
\end{remark}

\subsection{Connections between easy and non-easy group-theoretical quantum groups}

It follows from the Tannaka-Krein duality that
$\Homm_A(k,l) \subseteq \Homm_{C(S_n)} = \lspan \{T_p^{(n)} \mid p\in P(k,l)\}$ for all group-theoretical quantum groups $(A,u,n)$. Thus for any partition $p\in P(k,l)$ we can write the linear map $\widehat{T}_p^{(n)}$ as a sum of linear maps in $\{T_q^{(n)} \mid q\in P(k,l)\}$. In the following we describe this construction and show that the coefficients of these sums can be computed recursively.

\begin{definition}
For all $p\in P(k,l)$ we define 
\[ M_{\leq p} := \{ \ke(\textbf{i},\textbf{j}) \in P(k,l) \mid (\textbf{i},\textbf{j}) \in \sSym_{\infty}(\ind(p))\} .\]
\end{definition}

It is easy to see that the set $M_{\leq p}$ consists of all partitions $q\in P(k,l)$ which arise from $p$ by joining blocks. 

\begin{lemma} \label{lem::widehatT_T}
Let $p\in P(k,l)$ and $n\in \NN$. Then we have
\[ \widehat{T}_p^{(n)}= T_p^{(n)} - \sum_{q\in M_{\leq p}\backslash \{p\}} \widehat{T}_q^{(n)} .\]
In particular, $\widehat{T}_p^{(n)} \in \lspan \{T_q^{(n)} \mid q\in M_{\leq p}\}$ and $\widehat{T}_p^{(n)}=T_p^{(n)}$ if $p$ has just one block.
\end{lemma}

\begin{proof}
Let $(\textbf{i},\textbf{j}) \in \NN^k\times \NN^l$. By the definition of $M_{\leq p}$ we have $(\textbf{i},\textbf{j}) \in \sSym_{\infty}(\ind(p))$ if and only if there exists a partition $r\in M_{\leq p}$ with $r=\ke(\textbf{i},\textbf{j})$. By Def. \ref{def::ker_ind} this is equivalent to $(\textbf{i},\textbf{j}) \in \Sym_{\infty}(r)$ for some $r\in M_{\leq p}$ and hence to $\sum_{q\in M_{\leq p}} \mathbb{1}_{\Sym_{\infty}(\ind(q))}(\textbf{i},\textbf{j})=1$. 
Hence it follows that 
\begin{eqnarray*}
\widehat{\delta}_p(\textbf{i},\textbf{j}) 
&=&\mathbb{1}_{\Sym_{\infty}(\ind(p))}(\textbf{i},\textbf{j}) \\
&=&\mathbb{1}_{\sSym_{\infty}(\ind(p))}(\textbf{i},\textbf{j}) - \sum_{q\in M_{\leq p}\backslash \{p\}} \mathbb{1}_{\Sym_{\infty}(\ind(q))}(\textbf{i},\textbf{j}) \\
&=&\delta_p(\textbf{i},\textbf{j}) - \sum_{q\in M_{\leq p}\backslash \{p\}} \widehat{\delta}_q(\textbf{i},\textbf{j})
\end{eqnarray*}
for all $(\textbf{i},\textbf{j}) \in \NN^k\times \NN^l$.
Thus we have
\begin{eqnarray*}
\widehat{T}_p^{(n)} (e_{i_1} \ot \cdots \ot e_{i_k}) 
&=& \sum_{\textbf{j}\in \underline{n}^l}  \widehat{\delta}_p(\textbf{i},\textbf{j}) ~ e_{j_1} \ot \cdots \ot e_{j_l} \\
&=& \sum_{\textbf{j}\in \underline{n}^l}  (\delta_p(\textbf{i},\textbf{j}) - \sum_{q\in M_{\leq p}\backslash \{p\}} \widehat{\delta}_q(\textbf{i},\textbf{j})) ~ e_{j_1} \ot \cdots \ot e_{j_l} \\ 
&=& T_p^{(n)} (e_{i_1} \ot \cdots \ot e_{i_k}) \\
&&-\sum_{q\in M_{\leq p}\backslash \{p\}} \widehat{T}_q^{(n)} (e_{i_1} \ot \cdots \ot e_{i_k}) 
\end{eqnarray*}
for all $\textbf{i}=(i_1,\ldots,i_k) \in \underline{n}^k$.
\end{proof}

Now, we can show that our definitions and results are compatible with the definitions and results in the case of group-theoretical easy quantum groups by Raum and Weber.

\begin{lemma}
Let the group-theoretical CMQG $(A,u,n)$ induced by $A$ be easy. Then $\cR$ is a category of partitions and we have 
\[ \Homm_{A}(k,l) = \lspan \{ T^{(n)}_p \mid p\in  \cR(k,l) \}. \]
\end{lemma}

\begin{proof}
By Corollary \ref{cor::of_main_3} and the proof of Theorem \ref{ThMain1} we have $\cR =\{ \ke (\textbf{i},\textbf{j}) \mid a_{\textbf{i}} a_{\textbf{j}}^{-1} \in N_{\infty}\}$. To show that $\cR$ is a category of partitions we have to show that it is closed under composition. Let $p=\ke(\textbf{i},\textbf{j}) \in \cR(k,l)$ and $q=\ke(\textbf{f},\textbf{g})\in \cR(k',l')$.
For the composition of $p$ and $q$ we connect the lower points of $p$ with the upper points of $q$. We consider the partition $r\in P(l,0)$ consisting of those $l$ middle points in $qp$ and their strings we obtain by connecting $p$ and $q$. Then there exists a multi-index $\textbf{x}\in \NN^l$ and maps $\phi_1,\phi_2:\NN \to \NN$ such that $r=\ke(\textbf{x})$ and $\textbf{x}=(\phi_1(j_1),\ldots,\phi(j_l))=(\phi_2(f_1),\ldots,\phi_2(f_l))$. By Theorem \ref{thm::raumweber_thm4.4.} $N$ is $\sSym_n$-invariant and thus $N_{\infty}$ is $\sSym_{\infty}$-invariant. It follows that $a_{\phi_1(\textbf{i})} a_{\textbf{x}}^{-1}=\sigma_{\phi_1}(a_{\textbf{i}}a_{\textbf{j}}^{-1}) \in N_{\infty}$ and $a_{\textbf{x}} a_{\phi_2(\textbf{g})}^{-1}= \sigma_{\phi_2}(a_{\textbf{f}}a_{\textbf{g}}^{-1}) \in N_{\infty}$. This implies $s:=\ke(\phi_1(\textbf{i}),\textbf{x}) \in \cR$ and $t:=\ke(\textbf{x},\phi_2(\textbf{g})) \in \cR$. By construction $s$ and $t$ are compatible and we have $qp=ts \in \cR$. \\
Again let $p\in \cR(k,l)$. It is easy to see that $\cR$ is closed under joining blocks as group-theoretical category of partitions and hence we have $q\in \cR$ for all $q\in M_{\leq p}$. Thus it follows that $\widehat{T}^{(n)}_p \in \lspan \{T^{(n)}_{q} \mid q\in M_{\leq p}\} \subseteq \lspan \{ T^{(n)}_{r} \mid r\in  \cR(k,l) \}$ and $T^{(n)}_p\in \lspan \{ \widehat{T}^{(n)}_q \mid q\in M_{\leq p}\} \subseteq \lspan \{ \widehat{T}^{(n)}_r \mid r\in  \cR(k,l) \}$. It follows that $\Homm_{A}(k,l) = \lspan \{ T^{(n)}_p \mid p\in  \cR(k,l) \}$.
\end{proof}

Finally, we show that we can embed any non-easy group-theoretical quantum group in two group-theoretical easy quantum groups.

\begin{lemma} \label{lem::easy_embedding}
We define 
\[ N_1:=\{x\in N\mid sS_n(\{x\}) \subseteq N \} \text{ and } N_2:=\langle \langle sS_n(N) \rangle \rangle \unlhd \ZZ_2^{*n} . \] 
Then $N_1$ and $N_2$ are $\sSym_n$-invariant normal subgroups with $N_1 \leq N \leq N_2$. Hence $A_1:=C^*(\ZZ_2^{*n}/N_1) \Join C(\Sym_n)$ and $A_2:=C^*(\ZZ_2^{*n}/N_2) \Join C(\Sym_n)$ give rise to group-theoretical easy quantum groups with $A_2 \overset{\sim}{\twoheadrightarrow} A \overset{\sim}{\twoheadrightarrow} A_1$.
\end{lemma}

\begin{proof}
Obviously, $N_2$ is a $\sSym_n$-invariant normal subgroup with $N \leq N_2$. It is also easy to check that $N_1 \unlhd \ZZ_2^{*n}$ is a normal subgroup with $N_1\leq N$. Hence we have to show that $N_1$ is $\sSym_n$-invariant. Let $x\in N_1$. Then we have $sS_n(sS_n(\{x\}))=sS_n(\{x\})\subseteq N$ and hence $sS_n(\{x\}) \subseteq N_1$.
\end{proof}

\section{A concrete example of a non-easy group-theoretical quantum group}
We consider the group presented by
\begin{align*}
S:=\langle s_1,\ldots ,s_n \mid &s_i^2=1, (s_i s_j)^3=1,(s_bs_cs_bs_d)^2=1 \\ &\text{ for all } i,j,k,b,c,d\in \underline{n} \text{ with } |\{b,c,d\}|=3 \rangle ,
\end{align*}
which is isomorphic to $\Sym_{n+1}$ via $\varphi :S\to \Sym_{n+1},~s_i \mapsto (i~n+1)$ (see \cite{So94}). Obviously
\[ N_S := \langle \langle (a_i a_j)^3,(a_ba_ca_ba_d)^2 \mid i,j,k,b,c,d\in \underline{n} \text{ with } |\{b,c,d\}|=3 \rangle \rangle \unlhd \ZZ_2^{*n} \] 
is an $\Sym_n$-invariant normal subgroup. Since $\phi((a_1a_2a_1a_3)^2)=(a_1a_2)^4 \notin N_S$ for $\phi:\underline{n}\to \underline{n}$ with $\phi(\{2,3\})=\{2\}$ and $\phi(x)=x$ for all $x\in \underline{n}\backslash \{2,3\}$ it follows that $N_S$ is not $\sSym_n$-invariant. Thus $A_\Sym:=C^*(Z_2^{*n}/N_S) \Join C(\Sym_n)$ is a non-easy group-theoretical quantum group. \\
By Remark \ref{rem::thm3.1.} we have
\begin{align*}
A_S &\cong C^*(u_{ij}, 1\leq i,j\leq n\mid u_{ij}^*=u_{ij}, \sum_{\textbf{f}=(f_1,\ldots,f_k) \in \underline{n}^k} u_{f_1i_1}\cdot \ldots \cdot u_{f_ki_k}=1 \\
&\hspace{3.8cm} \text{ for all } k\in \NN_0, \textbf{i}\in \underline{n}^k, a_{\textbf{i}}\in N_S) .
\end{align*}
By Lemma \ref{lemma::hatp_relation} a relation $\sum_{\textbf{f}=(f_1,\ldots,f_k) \in \underline{n}^k} u_{f_1i_1}\cdot \ldots \cdot u_{f_ki_k}=1$ for some $k\in \NN_0, \textbf{i}\in \underline{n}^k, a_{\textbf{i}}\in N_S$ is equivalent to $\widehat{T}^{(n)}_{\ke (\textbf{i})} \in \text{Hom}_A(k,0)$ and by the proof of Theorem \ref{ThMain3} it is equivalent to $\ke (\textbf{i}) \in \cR$ for the corresponding skew category of partitions $\cR$. Hence by the proof of Theorem \ref{ThMain1} we see that it suffices to take the generating relations $N_{\text{gen}}=\{a_i^2,(a_ia_j^2)^2,(a_i a_j)^3,(a_ba_ca_ba_d)^2 \mid i,j,k,b,c,d\in \underline{n} \text{ with } |\{b,c,d\}|=3 \}$ for the presentation of $A_S$
\begin{align*}
A_S &\cong C^*(u_{ij}, 1\leq i,j\leq n\mid u_{ij}^*=u_{ij}, \sum_{\textbf{f}=(f_1,\ldots,f_k) \in \underline{n}^k} u_{f_1i_1}\cdot \ldots \cdot u_{f_ki_k}=1 \\
&\hspace{3.8cm} \text{ for all } k\in \NN_0, \textbf{i}\in \underline{n}^k, a_{\textbf{i}}\in N_{\text{gen}}) .
\end{align*}
By some easy calculations and Lemma \ref{lem::u_iju_ik=0} we obtain
\begin{align*}
A_\Sym \cong C^*(u_{ij}, 1\leq i,j\leq n\mid &u_{ij}^*=u_{ij},~u \text{ orthogonal},~u_{ij}^2 \text{ central} \\
&\text{projections}, u_{ij}u_{kl}u_{ij}=u_{kl}u_{ij}u_{kl}, \\
& \sum_{e,f,h} (u_{be} u_{cf} u_{be} u_{dh})^2=1 \text{ for } |\{b,c,d\}|=3) .\\
\end{align*}
We define the partitions 
\begin{align*}
&e_{3,3}:=\ke((1,1,1),(1,1,1)) \in P(3,3),\\
&e_{6,0}:=\ke((1,1,1,1,1,1)) \in P(6,0),\\
&e_{8,0}:=\ke((1,1,1,1,1,1,1,1)) \in P(8,0),\\
&h_3:=\ke((1,2,1,2,1,2)) \in P(6,0),\\
&r:=\ke((1,2,1,3,1,2,1,3)) \in P(8,0),\\
&r_1:=\ke((1,2,1,2,1,2,1,2)) \in P(8,0),\\
&r_2:=\ke((1,1,1,3,1,1,1,3)) \in P(8,0),\\
&r_3:=\ke((1,2,1,1,1,2,1,1)) \in P(8,0),\\
\end{align*}
By Theorem \ref{ThMain3} and Corollary \ref{cor::of_main_3} we have $\Hom_A(k,l)=\lspan \{\widehat{T}_p^{(n)} \mid p\in \cR\}$ with $\cR =\langle \ke(\textbf{i}) \mid a_{i_1}\cdot \ldots \cdot a_{i_k} \in N \rangle_{skew} =\langle \primarypart, h_3, r \rangle_{skew}$.\\
Moreover, by Lemma \ref{lem::widehatT_T} we have we have 
\begin{align*}
&T^{(n)}_{\widehat{p}} = T^{(n)}_p - T^{(n)}_{e_{3,3}} \text{ for } p=\primarypart,\\
&T^{(n)}_{\widehat{h_3}} = T^{(n)}_{h_3} - T^{(n)}_{e_{6,0}} \text{ and}\\
&T^{(n)}_{\widehat{r}} = T^{(n)}_r - T^{(n)}_{r_1} - T^{(n)}_{r_2} - T^{(n)}_{r_3} + 2T^{(n)}_{e_{8,0}}.
\end{align*}
The $\sSym_n$-invariant normal subgroups of Theorem \ref{lem::easy_embedding} are given by
\begin{align*}
&N_1 = \langle \langle a_i^2, (a_i a_j)^3, (a_i a_j a_i a_k)^6 \mid i,j,k\in \underline{n} \rangle \rangle \unlhd \ZZ_n^{*n},\\
&N_2 = \langle \langle a_i^2, a_i a_j \mid i,j,\in \underline{n} \rangle \rangle \unlhd \ZZ_n^{*n}.
\end{align*}
Thus we obtain 
\[ C(\ZZ_2 \rtimes S_n) \cong  C^*(\ZZ_2^{*n}/N_2) \Join C(\Sym_n) \overset{\sim}{\twoheadrightarrow} A_S \overset{\sim}{\twoheadrightarrow}  C^*(\ZZ_2^{*n}/N_1) \Join C(\Sym_n),\]
where $\overset{\sim}{\twoheadrightarrow}$ should be understood as a surjective *-homomorphism preserving the fundamental corepresentation as in Definition \ref{def::quantum_subgroup}.

\section{Appendix}
We correct the error in Theorem 4.4 of \cite{RW15}. Note that for any $n\in \NN \cup \{\infty\}$ and category of partitions $\cC$ the definition of Raum and Weber of 
\begin{align*}
F_{\infty}(\cC):=&~ \{ a_{\textbf{i}} \mid k\in \NN_0, p\in \cC(k,0), \delta_p(\textbf{i},0)\} \\
=&~ \sSym_n(\{ a_{\ind(p)} \mid k\in \NN_0, p\in \cC(k,0)\})
\end{align*}  
slightly differs from our definition 
\[ F_{\infty}(\cC)= \Sym_n(\{ a_{\ind(p)} \mid k\in \NN_0, p\in \cC(k,0)\}) .\]
But in the following we will consider group-theoretical categories of partitions for which these definitions coincide as group-theoretical categories of partitions are closed under connecting blocks.

\begin{theorem}{\cite[Thm. 4.4]{RW15}}
For all group-theoretical categories of partitions $\cC$ and all $n \in \NN \cup \{\infty\}$, the subgroup $F_n(\cC) \unlhd \ZZ_2^{*n}$ is an $\sSym_n$-invariant, normal subgroup. Vice versa, for every $n \in \NN \cup \{\infty\}$ and every $\sSym_n$-invariant, normal subgroup $N \unlhd \ZZ_2^{*n}$ there is a group-theoretical category of partitions $\cC$ such that $F_n(\cC) = N$.
\end{theorem}

\begin{proof}
The proof of the first statement is correct in \cite{RW15} so we have to show that all $\sSym_n$-invariant, normal subgroups $N \unlhd \ZZ_2^{*n}$ arise as $N = F_n(\cC)$ for some group-theoretical category of partitions $\cC$. \\
Raum and Weber put 
\[\cC := \{p \in P \mid p \text{ is a rotation of } \ke(\textbf{i}) \text{ for some } k \in \NN, a_{\textbf{i}} \in N\} \]
which is not a category of partitions as it is not closed under taking tensor products. Consider $\ke(\textbf{i}),\ke(\textbf{i'}) \in \cC(k,0)$. Raum and Weber assumed that $\{i_1, \dotsc, i_k\} \cap \{i'_1, \dotsc, i'_{k'}\} = \emptyset$ to conclude $\ke(\textbf{i})\ot \ke(\textbf{i'}) = \ke(\textbf{i}\textbf{i'})$ which is not always possible as $n\in \NN \cup \{\infty\}$ is fixed. Thus the partition $\ke(\textbf{i})\ot \ke(\textbf{i'})$ could have more than $n$ blocks and hence it would not be in $\cC$. \\
We can fix the proof as follows. 
The category of partitions we are looking for is
\[\cC = \{p \in P \mid p \text{ is a rotation of } \ke(\textbf{i}) \text{ for some } k \in \NN, a_{\textbf{i}} \in N_{\infty}\} ,\]
where $N_\infty$ is the closure of $N$ in $\ZZ_2^{*\infty}$, see Definition \ref{def::Ninfty}.
The proof that $\cC$ is a group-theoretical category of partitions works exactly as in the proof of Theorem 4.4. in \cite{RW15}. Moreover, the proof of Raum and Weber also shows that $F_\infty(\cC) = N_{\infty}$ and hence we have $F_n(\cC) = F_\infty(\cC) \cap \ZZ_2^{*n} = N_{\infty} \cap \ZZ_2^{*n}$.\\
It remains to show that $N_{\infty} \cap \ZZ_2^{*n}=N$. We have $N\subseteq N_\infty \cap \ZZ^{*n}$, so let $a_{\textbf{i}} \in N_\infty \cap \ZZ^{*n}$. By definition of $N_{\infty}$ there exist $k\in \NN$, $a_{\textbf{g}_1}, \ldots ,a_{\textbf{g}_k} \in S_{\infty} (N)$ and invertible elements $a_{\textbf{f}_1}, \ldots ,a_{\textbf{f}_k} \in \ZZ^{*\infty}$ with 
$$ a_{\textbf{i}} = \prod_{j=1}^k a_{\textbf{f}_j} a_{\textbf{g}_j} a_{\textbf{f}_j}^{-1} .$$
Since the element $a_{\textbf{i}}$ is by assumption a word in the letters $a_1,\ldots ,a_n$, it is invariant under the map $\phi \in \sSym_\infty$ with $\phi (x)=x$ for all $x\in \underline{n}$ and $\phi (x) = 1$ for all $x\in \NN \backslash \underline{n}$. Thus 
$$ a_{\textbf{i}} = \phi (a_{\textbf{i}}) = \phi \left( \prod_{j=1}^k a_{\textbf{f}_j} a_{\textbf{g}_j} a_{\textbf{f}_j}^{-1} \right) = \prod_{j=1}^k a_{\phi (\textbf{f}_j)} a_{\phi (\textbf{g}_j)} a_{\phi (\textbf{f}_j)}^{-1} .$$
Then $a_{\phi (\textbf{f}_j)} \in \ZZ^{*n}$ for all $j\in \underline{k}$ and since $N$ is $\sSym_n$-invariant, we also have $a_{\phi (\textbf{g}_j)} \in N$ for all $j\in \underline{k}$. As $N$ is normal, it follows that  $a_{\textbf{i}}\in N$.
\end{proof}

\begin{remark}
The set Raum and Weber used was still a generating set for the category of partitions we are looking for. Let $n \in \NN \cup \{\infty\}$ and let $N \unlhd \ZZ_2^{*n}$ be an $\sSym_n$-invariant, normal subgroup. We put
\begin{align*}
&\cC := \{p \in P \mid p \text{ is a rotation of } \ke(\textbf{i}) \text{ for some } k \in \NN, a_{\textbf{i}} \in N_{\infty}\} \text{ and} \\
&\cC' := \langle \ke(\textbf{i}) \mid a_{\textbf{i}} \in N \rangle .
\end{align*} 
Then we have $\cC=\cC'$.
\end{remark}

\begin{proof}
By the previous theorem $\cC$ is a category of partitions with $\{\ke(\textbf{i}) \mid a_{\textbf{i}} \in N\} \subseteq \cC$ and thus we have $\cC' \subseteq \cC$. By definition of $\cC'$ we have $N\subseteq F_n(\cC') \subseteq F_{\infty}(\cC')$. It also follows from the previous theorem that $F_{\infty}(\cC')$ is an $\sSym_{\infty}$-invariant, normal subgroup and hence we have $N_{\infty} \subseteq F_{\infty} (\cC')$ by the definition of $N_{\infty}$. It follows that $F_{\infty}(\cC)=N_{\infty} \subseteq F_{\infty}(\cC')$ and hence $\cC \subseteq \cC'$.
\end{proof}

\end{document}